\documentclass[12pt]{article}

\usepackage{graphicx}%
\usepackage{multirow}%
\usepackage{amsmath,amssymb,amsfonts}%
\usepackage{amsthm}%
\usepackage{mathrsfs}%
\usepackage[title]{appendix}%
\usepackage{xcolor}%
\usepackage{textcomp}%
\usepackage{manyfoot}%
\usepackage{booktabs}%
\usepackage{algorithm}%
\usepackage{algorithmicx}%
\usepackage{algpseudocode}%
\usepackage{listings}%
\usepackage{geometry}
\usepackage{tikz}
\newcommand*\circled[1]{\tikz[baseline=(char.base)]{
\node[shape=circle, draw, inner sep=0.6pt] (char) {#1};}
}
\newcommand\kronF[2]{#1^{\circled{\tiny{#2}}}}
\theoremstyle{plain}
\newtheorem{theorem}{Theorem}

\newtheorem{lemma}[theorem]{Lemma}
\theoremstyle{plain}

\newtheorem{remark}{Remark}

\theoremstyle{plain}
\newtheorem{definition}{Definition}

\usepackage{authblk}
\usepackage{cite}
\newcommand\degd{{\rm d}}

\begin{document}

\title{Energy function approximations for differential algebraic polynomial systems of Stokes-type}

\author[1]{Hamza Adjerid}
\author[1]{Jeff Borggaard}

\affil[1]{Department of Mathematics, Virginia Tech, Blacksburg, VA USA}

\date{}                     
\setcounter{Maxaffil}{0}
\renewcommand\Affilfont{\itshape\small}

\maketitle

\abstract{Energy functions are generalizations of controllability and observability Gramians to nonlinear systems and as such find applications in both nonlinear balanced truncation and feedback control.  These energy functions are solutions to the Hamilton-Jacobi-Bellman (HJB) equations---partial differential equations defined over spatial dimensions determined by the number of state variables in the nonlinear system.  Thus, they cannot be resolved with local basis functions for problems of even modest dimension.  In this paper, we extend recent results that utilize Kronecker products to generate polynomial approximations to HJB equations.  Specifically, we consider the addition of linear drift terms that exhibit a Stokes-type differential-algebraic equation (DAE) structure.  This extension leverages the so-called {\em strangeness framework} for DAEs to create separate sets of algebraic and differential variables with differential equations that only involve the differential variables and algebraic equations that relate them both.  At this point, the existing polynomial approximations using Kronecker products can be applied to find approximations to the energy functions.  This approach is demonstrated for two polynomial feedback control problems.  The standard transformation destroys the sparsity in the original system.  Thus, we also present a formulation that preserves the original sparsity structure.}

\maketitle

\section{Introduction}\label{sec1}

Energy functions generalize the notions of controllability and observability to nonlinear systems.  As such, they find natural applications in control and balanced truncation for these systems.  
Consider the nonlinear, control affine input-output system
\begin{equation}
\label{eq:nl_system}
    \dot{x}(t)=f(x(t))+Bu(t), \qquad y(t)=Cx(t),
\end{equation}
where $f:\mathbb{R}^n\rightarrow \mathbb{R}^n$ is a smooth function, $B\in \mathbb{R}^{n\times m}$,
and $C\in \mathbb{R}^{p\times n}$ and initial condition $x(0)=x^0$.   Given the space of admissible controls as ${\cal U}^- = L_2(-\infty,0;\mathbb{R}^m)$ and ${\cal U}^+ = L_2(0,\infty;\mathbb{R}^m)$,
then for any $\eta\leq 1$, we define~\cite{vanderschaft1992L2gainAnalysisNonlinear}
\begin{equation}
\label{eq:pastEnergy}
{\mathcal{E}}_\eta^{-}(x^0)    :=\min_{\substack{u \in {\cal U}^- \\ x(-\infty) = 0 \\ x(0) = x^0}} \ \frac{1}{2} \int_{-\infty}^{0} \eta\Vert {y}(t) \Vert^2    +    \Vert {u}(t) \Vert^2 {\rm{d}}t
\end{equation}
as the {\em past energy}.    The {\em future energy} is defined for three cases.    For $0<\eta\leq 1$, 
\begin{equation}
\label{eq:futureEnergy1}
{\mathcal{E}}_\eta^{+}({    {x}}^0)    :=\min_{\substack{{        u} \in {\cal U}^+ \\ {    {x}}(0) = {    {x}}^0 \\ {    {x}}(\infty) = {        0}}} \ \frac{1}{2} \int_{0}^{\infty} \Vert {        y}(t) \Vert^2    +    
\frac{1}{\eta}\Vert {        u}(t) \Vert^2 {\rm{d}}t;
\end{equation}
for $\eta=0$,
\begin{equation}
\label{eq:controllability}
{\mathcal{E}}_\eta^{+}({    {x}}^0) := \frac{1}{2} \int_0^\infty \Vert {        y}(t) \Vert^2{\rm{d}}t, \quad {    {x}}(0)={    {x}}^0, \ \mbox{and}\ {        u}(t)\equiv {        0};
\end{equation}
and for $\eta<0$,
\begin{equation}
\label{eq:futureEnergy3}
{\mathcal{E}}_\eta^{+}({    {x}}^0)    :=\max_{\substack{{        u} \in {\cal U}^+ \\ {    {x}}(0) = {    {x}}^0, \\    {    {x}}(\infty) = {        0}}} \ \frac{1}{2} \int\displaylimits_{0}^{\infty} \Vert {        y}(t) \Vert^2    +    \frac{1}{\eta}\Vert {        u}(t) \Vert^2 {\rm{d}}t.
\end{equation}
These are often defined using an $L_2$-gain parameter $\gamma_0$ when $\eta = -\gamma_0^{-2}$ \cite{vanderschaft1992L2gainAnalysisNonlinear} or using an $H_\infty$-gain parameter $\gamma>0$ when $\eta = 1-\gamma^{-2}$ \cite{scherpen1996Hinf}.
In the $H_\infty$ setting, when we take the limit $\gamma\rightarrow 1$, $\eta=0$, we recover that ${\mathcal{E}}_0^{-}$ is the controllability energy and ${\mathcal{E}}_0^{+}$ is the observability energy.    Additionally, in the limit $\gamma\rightarrow \infty$, we obtain $\eta=1$ and recover the energy functions used in HJB balancing \cite{scherpen1994normalized}.  Finally, when $\eta=-1$, the past and future energy functions agree.  These connections are detailed in \cite{kramer2023nonlinear1}.    
Furthermore, the energy functions can be characterized by HJB equations, cf.~\cite{scherpen1996Hinf}.    If $\bar{\mathcal E}$ is a solution to
\begin{equation} \label{eq:HJB-NLHinfty2}
 0    =    \frac{\partial \bar{\mathcal{E}}({    {x}})}{\partial {    {x}}} {        f}({    {x}}) + \frac{1}{2}    \frac{\partial \bar{\mathcal{E}}({    {x}})}{\partial {    {x}}} {        B} {        B}^\top     \frac{\partial^\top     \bar{\mathcal{E}}({    {x}})}{\partial {    {x}}}
 - \frac{\eta}{2} {        x}^\top     {        C}^\top        {        Cx}
\end{equation}
with $\bar{\mathcal{E}}({        0}) = 0$ and ${        0}$ is an asymptotically stable fixed point of 
\begin{equation}
    \dot{        x} =    - \left ( {        f}({    {x}}) +{        B}{        B}^\top     \frac{\partial^\top     \bar{\mathcal{E}}({    {x}})}{\partial {    {x}}} \right ), 
\end{equation}
then $\bar{\mathcal{E}}({        x})$ is the past energy function ${\mathcal{E}}_\eta^{-}({    {x}})$ from \eqref{eq:pastEnergy}.
Likewise, if $\tilde{\mathcal{E}}$ is a solution to
\begin{equation} \label{eq:HJB-NLHinfty1}
0    = \frac{\partial \tilde{\mathcal{E}}({    {x}})}{\partial {    {x}}} {        f}({    {x}})
 - \frac{\eta}{2}    \frac{\partial \tilde{\mathcal{E}}({    {x}})}{\partial {    {x}}} {        B} {        B}^\top     \frac{\partial^\top     \tilde{\mathcal{E}}({    {x}})}{\partial {    {x}}}
 + \frac{1}{2}{        x^\top}{        C}^\top     {        C}{        x}
\end{equation}
with $\tilde{\mathcal{E}}({        0}) = 0$ and ${        0}$ is an asymptotically stable fixed point of
\begin{equation}
     \dot{        x} =    {        f}({    {x}}) - \eta{        B} {        B}^\top     \frac{\partial^\top     \tilde{\mathcal{E}}({    {x}})}{\partial {    {x}}}, 
\end{equation}
then this solution $\tilde{\mathcal E}({        x})$ is the future energy function ${\mathcal{E}}_\eta^{+}({    {x}})$ defined in (\ref{eq:futureEnergy1})-(\ref{eq:futureEnergy3}). 

Polynomial approximations to the solutions of (\ref{eq:HJB-NLHinfty2}) and (\ref{eq:HJB-NLHinfty1}) were developed in \cite{kramer2023nonlinear1}.  These leverage the developments of Al'Brekht~\cite{albrekht1961optimal}, Lukes~\cite{lukes1969optimal}, Navasca and Krener~\cite{navasca2000solution}, and the recent exploitation of Kronecker products in \cite{dolgov2019tensor,kalise2018polynomial,breiten2019Taylor,borggaard2019QQR,borggaard2021PQR} to find degree $\degd$ approximations of the form
\begin{align}
\label{eq:pastEpoly}
\mathcal{E}_\eta^{-}(    {x})\approx&\cfrac{1}{2}\left(     {v}_2^\top \kronF{    {x}}{2}+    {v}_3^\top \kronF{    {x}}{3}+\cdots +    {v}_\degd^\top \kronF{    {x}}{\degd}\right),\\
\label{eq:futureEpoly}
\mathcal{E}_\eta^{+}(    {    {x}})\approx&\cfrac{1}{2}\left(     {w}_2^\top \kronF{    {x}}{2}+    {w}_3^\top \kronF{    {x}}{3}+\cdots +    {w}_\degd^\top \kronF{    {x}}{\degd}\right),
\end{align}
where $\kronF{x}{\degd}$ is a Kronecker product involving $\degd$ copies of the vector $    {x}$, e.g. $\kronF{x}{3} =     {x}\otimes    {x}\otimes    {x}$.  For uniqueness, we require the coefficients in (\ref{eq:pastEpoly})-(\ref{eq:futureEpoly}) to be symmetrized, ensuring that equivalent monomial terms have the same coefficient.  In this case, the coefficients ${        v}_k$ and ${        w}_k$ are solutions to structured linear systems that enable efficient computation; cf.~\cite{chen2019RecursiveBlockedAlgorithms,borggaard2021PQR}.    
This paper extends this development to the case of differential algebraic equations (DAEs).

DAEs (also known as descriptor systems or singular systems) arise in many modeling problems.   This is especially common when systems are coupled through algebraic constraints.  Here, we consider an important class of input-output DAE systems with a so-called saddle-point or Stokes-type structure,
\begin{equation}
\label{eq:NaSt_DAE}
\begin{split}
    \begin{pmatrix}
            {E}_{11} &     {0}\\
            {0} &     {0}
    \end{pmatrix} 
    \begin{pmatrix}
        \dot{    {x}}_1 \\
        \dot{    {x}}_2 
    \end{pmatrix} &=
    \begin{pmatrix}
            {A}_{11} &     {A}_{12}\\
            {A}_{12}^\top &     {0}
    \end{pmatrix} 
    \begin{pmatrix}
            {x}_1 \\
            {x}_2 
    \end{pmatrix}
+\begin{pmatrix}
             {N} \\
             {0}    
    \end{pmatrix}
            \left({x}_1 \otimes     {x}_1\right)+
    \begin{pmatrix}
         {B}_1 \\
            {B}_2 
    \end{pmatrix}    {u},\\
    {y}=&~    {C}_1    {x}_1,    
\end{split}
\end{equation}
{  with initial condition $x_1(0)={x}^0 \in \mathbb{R}^{n_1}$,} where $x_1(t)\in\mathbb{R}^{n_1}$, $x_2(t)\in\mathbb{R}^{n_2}$, $u(t)\in\mathbb{R}^m$, and $y(t)\in\mathbb{R}^p$.  All of the matrices above have compatible dimensions for the block structure indicated.  The matrix $E_{11}\in\mathbb{R}^{n_1\times n_1}$ is invertible, $A_{12}\in\mathbb{R}^{n_1\times n_2}$ is full-rank, and we restrict our attention to linear algebraic constraints.  Thus, $N\in\mathbb{R}^{n_1\times n_1^2}$.  Systems of this form appear when imposing linear constraints (here $A_{12}^\top x_1 + B_2u=0$) on quadratic differential equations.  They arise when discretizing boundary control problems~\cite{benner2015timedependent}, in domain decomposition methods~\cite{chan1994domain}, and in discretized flow control problems involving the Stokes or Navier-Stokes equations~\cite{gunzburger1989finite,fursikov1991optimal,borggaard2010linear}.

This paper develops approximations of energy functions and the associated feedback control laws for systems in the form (\ref{eq:NaSt_DAE}) with a quadratic running cost such as the integrand in (\ref{eq:futureEnergy1}).  The development of feedback control for linear DAEs has a rich history that includes the work in the 1980s by Cobb~\cite{cobb1983descriptor,cobb1984controllability}, Lewis~\cite{lewis1985optimal}, and Bender and Laub~\cite{bender1987linear}.  Further connections between DAEs and control theory were made in later works that include~\cite{kunkel2008OptimalControlUnstructured,Sjoberg,heiland2016DAE,ahmad2017MomentmatchingBasedModel} among several others.

To take advantage of recent computational advances in the approximation of feedback controllers for polynomial systems, we leverage the strangeness framework by Kunkel and Mehrmann~\cite{KunkelandMehrmann} to transform the input-output system (\ref{eq:NaSt_DAE}) to a system of ODEs with quadratic nonlinearities.  This same approach is easily extended to higher-degree nonlinearities in $x_1$ as in \cite{borggaard2021PQR}, but we restrict our attention to the quadratic case to simplify the discussion.  We apply the approximation results developed in \cite{kramer2023nonlinear1} to approximate energy functions associated with Stokes-type DAEs using the index reduction strategy in \cite{HeinkandSorenProj}.  We will present formulations for both cases $B_2=0$ and $B_2\neq 0$.  {  We note that these index reduction approaches usually destroy sparsity in these nonlinear systems. Thus, we present a reformulation of these linear systems that preserve sparsity and enable the matrix-free approximations to the energy function coefficients. Finally, we present numerical results that demonstrate the effectiveness of the approximations to the energy functions by using the latter to develop polynomial feedback controllers and comparing numerically integrated costs with the energy function approximations evaluated at the initial condition. }

\section{Background}\label{sec2}

\subsection{Overview of polynomial approximations for HJB\label{sec:II.B}}
Polynomial approximations to energy functions in the forms defined above (\ref{eq:pastEnergy})--(\ref{eq:futureEnergy3}) were developed for nonlinear balancing~\cite{fujimoto2008ComputationNonlinearBalanced}.    
In \cite{kramer2023nonlinear1}, the authors developed a scalable approach to approximate energy functions by polynomials when they are written using Kronecker products.  To review their approach, we define the {\em generalized $k$-way Lyapunov matrix} (or a special \textit{Kronecker sum} \cite{benzi2017approximation}) 
for $M\in\mathbb{R}^{q\times n}$ and $E\in\mathbb{R}^{n\times n}$ as
\begin{equation}\label{eq:calL}
        \mathcal{L}_k^E(M)\! :=\! \underbrace{M \otimes E\otimes \cdots \otimes E}_{k \    \text{times}} + \cdots + \underbrace{E \otimes \cdots \otimes E \otimes M}_{k \ \text{times}} \in \mathbb{R}^{n^{k-1}q \times n^k},
\end{equation}
where we drop the superscript if $E$ is $I_n$, the $n$-dimensional identity matrix. 
We present the results from \cite{kramer2023nonlinear1} for calculating the coefficients ${v}_k$ and ${w}_k$ in (\ref{eq:pastEpoly})-(\ref{eq:futureEpoly}) when system (\ref{eq:nl_system}) is specialized to ${f}({        x})\equiv {Ax}+{        N}\kronF{x}{2}$, and the pairs $({        A},{        B})$ and $({        A},{        C})$ are stabilizable and detectable. 
\begin{theorem}[] {\rm (\cite[Th.~7]{kramer2023nonlinear1})} Consider the system (\ref{eq:nl_system}) with the conditions stated above.    Given $\eta\leq 1$, and past energy function $\mathcal{E}^-_\eta(    {x})$ expanded with the
coefficients $    {v}_i,$    $i = 2, \ldots, \degd$    in (\ref{eq:pastEpoly}). Then, $    {v}_2 = {\tt vec} (    {V}_2)$,
where $    {V}_2$ is the symmetric positive definite solution to
the $\mathcal{H_\infty}$ Riccati equation
\begin{equation}\label{eq:12}
 0=    {A}^\top         {V}_2+    {V}_2    {A}-\eta     {C}^\top         {C}+    {V}_2    {B}    {B}^\top         {V}_2.     
\end{equation}
Moreover, the coefficient vectors $    {v}_k = {\tt vec}(    {V}_k) \in \mathbb{R}^{n^k}$ for $3\leq k\leq \degd$ are the  symmetrized solutions of the linear systems
\begin{equation}
        \mathcal{L}_k((    {A} +     {B}    {B}^\top    {V}_2    )^\top)    {v}_k= 
         -\mathcal{L}_{k-1}(    {N}^\top    )    {v}_{k-1}-\cfrac{\eta}{4}\sum_{\substack{i,j>2\\ i+j=k+2}} ij\,{\tt vec}(    {V}_i^\top     {B}    {B}^\top     {V}_j).
\end{equation}
\end{theorem}

\begin{theorem}[] {\rm (\cite[Th.~6]{kramer2023nonlinear1})} Consider the system (\ref{eq:nl_system}) with the conditions stated above. Given $\eta\leq 1$ and future energy function $\mathcal{E}^+_\eta(    {x})$ expanded with the
coefficients ${        w}_i,$ $i = 2, \ldots, \degd$ in (\ref{eq:futureEpoly}). Then, $    {w}_2 = {\tt vec} (    {W}_2)$,
where $    {W}_2$ is the symmetric positive definite solution to
the $\mathcal{H_\infty}$ Riccati equation
\begin{equation}\label{eq:13}
 0=    {A}^\top        {W}_2+    {W}_2    {A}+    {C}^\top        {C}-\eta     {W}_2    {B}    {B}^\top     {W}_2.     
\end{equation}
Moreover, the coefficient vectors $    {w}_k = {\tt vec}(    {W}_k) \in \mathbb{R}^{n^k}$ for $3\leq k\leq \degd$ are the  symmetrized solutions of the linear systems
\begin{align*}
\begin{split}
\mathcal{L}_k((    {A} -\eta     {B}    {B}^\top    {W}_2 )^\top)    {w}_k=
\ -\mathcal{L}_{k-1}(    {N}^\top )    {w}_{k-1}+\cfrac{\eta}{4}\sum_{\substack{i,j>2\\ i+j=k+2}} ij \, {\tt vec}(    {W}_i^\top        {B}    {B}^\top        {W}_j).  
\end{split}
\end{align*}
\end{theorem}
The proofs for both theorems can be found in \cite{kramer2023nonlinear1}. Although these systems grow exponentially in $k$ and are polynomial in $n$, the authors developed an efficient implementation that takes advantage of the Kronecker structure of the systems based on the work in \cite{chen2019RecursiveBlockedAlgorithms,borggaard2019QQR,borggaard2021PQR}. 
As shown in \cite{kramer2023nonlinear1}, polynomial approximations are accurate near the origin. %The main issue with these approximations is negativity away from the origin.    However, energy functions must be positive definite by definition. One way to overcome this issue is to propose function approximations that impose non-negative definiteness.

\subsection{DAE strangeness-index and index reduction}
\label{sec:II.D}
In most cases, systems described by a set of differential and algebraic equations are assumed to be equivalent to dynamical systems described only by a set of differential equations where the algebraic part is used to eliminate some variables from the initial set of differential equations. In those cases, the DAE system is turned into a state-space model with the form 
\begin{align*}
        \dot{    {x}}=    {f}(t,    {x},    {u})
\end{align*}
where $    {f:\ } I\times\mathbb{R}^n\times\mathbb{R}^m\rightarrow \mathbb{R}^n$, $t$ is time, $x$ are states and $    {u}$ is the control input. Although this model seems to be more convenient and easier to handle, it presents some disadvantages for some systems that are naturally described by sets of both algebraic and differential equations and when conversion to a state-space model is cumbersome, especially when the number of algebraic equations is large; cf.~\cite{petzold1982differential,Pantelides,Fritzon}.  Computational challenges arise when it is not easy to distinguish between the dynamical states and the algebraic states. The concept of a DAE index was introduced to measure the difficulty encountered when solving a DAE. The most common index is the {\em differential index} $\nu$ that measures how far a DAE is from an equivalent ODE obtained by differentiating the algebraic constraints. Another index related to the differential index is the {\em strangeness index} $\mu$ \cite{KunkelandMehrmann}, which we will denote ``s-index'' throughout the paper. The s-index measures how far the original DAE is from another DAE that decouples the algebraic variables from the differential equations and thus is easier to solve. A central concept of the s-index is invariance, namely, there is minimal sensitivity of the solution to equivalent reformulations.  In some cases, the differential index can share this property. However, the s-index can be defined for under-determined and over-determined systems and is thus a generalization of the differential index. If both the differential index and the s-index are defined, the relation between these two indices is $\mu=\text{max}\{0,\nu-1\}$. Consider a general DAE 
\begin{equation}
\label{eq:g_dae}
    {F}(t,    {x},\dot{    {x}},    {u})=0,
\end{equation} where $    {x} \in \mathbb{R}^n$, $    {u} \in \mathbb{R}^m$ and $F$ is a smooth function of its arguments, along with the set of solutions
\begin{equation*}
\mathbb{L}_\mu=\{(t,    {x},    {x}_{\mu+1},    {u},    {u}_\mu)\ |\    {F}_{\mu}(t,    {x},    {x}_{\mu+1},    {u},    {u}_\mu)=0\},
\end{equation*}
where $    {x}_{\mu+1}=(\dot{    {x}},\dots,    {x}^{(\mu+1)})$ , $    {u}_{\mu}=(\dot{    {u}},\dots,    {u}^{(\mu)})$ and $    {F}_{\mu}(t,    {x},    {x}_{\mu+1},    {u},    {u}_\mu)$ is the derivative array defined as \begin{equation*}
    {F}_{\mu}=\begin{pmatrix}
            {F}(t,    {x},\dot{    {x}},    {u})\\
        \frac{d}{dt}    {F}(t,    {x},\dot{    {x}},    {u})\\
        \vdots\\
        \frac{d^{\mu}}{dt^{\mu}}    {F}(t,    {x},\dot{    {x}},    {u})\\
\end{pmatrix}        .
\end{equation*}
We further use the notation of~\cite{KunkelandMehrmann} to define $F_{\mu;p}$ as the Jacobian of $F_\mu$ with respect to the set of variables in the collection $p$, for example $F_{\mu;x,\dot{x}} = [\frac{\partial}{\partial x} F_\mu~~ \frac{\partial}{\partial \dot{x}}F_\mu]$.  The following hypothesis is used to define the s-index.
\smallskip

\textbf{Hypothesis 1:}~\cite{KunkelandMehrmann} There exist integers $\mu$, $r$, $a$, $d$, and $v$ such that $\mathbb{L}_\mu$ is nonempty and for every %$    {z}_{\mu,0}$
$(t_0,    {x}^0,\dot{    {x}}^0,\cdots,    {x}^{0,(\mu+1)},u,u_\mu)\in \mathbb{L}_{\mu}$, there exists a neighborhood in which the following properties hold.
\begin{enumerate}
        \item The set $\mathbb{L}_\mu$ forms a manifold of dimension $(\mu+2)n+1-r+(\mu+1)m.$
        \item $\text{rank} (    {F}_{\mu;    {x},    {x}_{\mu+1}})=r$ on $\mathbb{L}_\mu$.
        \item $\text{nullity}(    {F_{\mu;    x,    x_{\mu+1}}}^\top)- \text{nullity}(    {{F}_{\mu-1;    {x},    {x}_{\mu+1}}}^\top)=v$ on $\mathbb{L}_\mu$ with $\text{nullity}({   {F}_{-1;    {x}}}^\top)\equiv 0.$
        \item $\text{rank}(    {F}_{\mu;    {x}_{\mu+1}})=r-a$ on $\mathbb{L}_u$, such that there exists smooth full-rank matrix functions $    {Z}_2 \in \mathbb{R}^{(\mu+1)n\times a} $ and $    {T}_2\in \mathbb{R}^{n \times n-a}$ defined on $\mathbb{L}_\mu$ satisfying
        \begin{itemize}
                \item $    {Z}_2^\top    {F}_{\mu;    {x}_{\mu+1}}=0$
                \item $\text{rank}(    {Z}_2^\top    {F}_{\mu;    {x}})=a$
                \item $    {Z}_2^\top    {F}_{\mu;    {x}}    {T}_2=0$.
        \end{itemize}
     \item $\text{rank}(    {F}_{\mu;    {\dot{x}}}    {T}_2)=d=n-a-v$ on $\mathbb{L}_\mu$, such that there exists a smooth full-rank matrix function $    {Z}_1\in \mathbb{R}^{n\times d}$ that satisfies $\text{rank}(    {Z}_1^\top    {F}_{    {\dot{x}}}    {T}_2)=d$.
\end{enumerate}
We assume that the ranks of these matrices are constant on the manifold $\mathbb{L}_{\mu}$.
\begin{definition}
The {\em s-index} of (\ref{eq:g_dae}) is the smallest non-negative integer $\mu$ such that
Hypothesis 1 is satisfied. A system is called {\em s-free} if it satisfies Hypothesis 1 with $\mu=0$.
\end{definition}
Hypothesis 1 is also constructive in the sense that it provides $d$, the dimension of the differential part of the system, $a$, the dimension of the algebraic part of the system, and $v$ the number of redundant equations when the system is overdetermined. Therefore, there exists a local s-free model with an implicit ODE with $d$ differential equations and $a$ implicit algebraic constraints \cite{Sjoberg}. It was also shown in \cite{KunkelandMehrmann} that if (\ref{eq:g_dae}) is smooth enough and Hypothesis 1 is also satisfied for $\mu+1$, then we can explicitly write the local s-free model as
\begin{equation} 
\label{eq:R_dae}
\begin{split}
    {\dot{x}}_1&={f}_1(t,{x}_1,{u})\\
0~&={f}_2(t,{x}_1,{x}_2,{u})    
\end{split}
\end{equation}
{  where ${f}_1:\mathbb{ I}\times\mathbb{R}^d\times\mathbb{R}^m\rightarrow \mathbb{R}^d \text{ and }{f}_2:\mathbb{ I}\times\mathbb{R}^d\times\mathbb{R}^a\times\mathbb{R}^m\rightarrow \mathbb{R}^a$, and the solution of (\ref{eq:R_dae}) is also a solution of (\ref{eq:g_dae}).
We are interested in the class of optimal control problems with infinite time horizon under DAE constraints. A crucial result from \cite{KunkelandMehrmann} and \cite{Sjoberg} is that under the following conditions, the closed-loop system will also be s-free with a unique solution.}

\begin{itemize}
        \item System (\ref{eq:R_dae}) is s-free for arbitrary $    {u} \in \Omega_u$ and given $x_1$, the equation $    {f}_2(t,    {x}_1,    {x}_2,    {u})=0$ can be uniquely solved for ${x}_2$ in \begin{align*}
        \Omega=\{    {x}_1\in\Omega_{    {x}_1}, \    {x}_2\in\Omega_{    {x}_2}, \    {u}\in\Omega_{    {u}}\ |\    {f}_2(t,    {x}_1,    {x}_2,    {u})=    {0}\},
\end{align*}
where $\Omega_{    {x}_1}$, $\Omega_{    {x}_2}$  and   $\Omega_{    {u}}$ are open and connected sets around an equilibrium state $0$.
\item The Jacobian of ${f}_2(t,    {x}_1,    {x}_2,    {u})$ with respect to $    {x}_2$ is non-singular on $\Omega$.
\item Consistent initial conditions are considered.
\end{itemize}

\smallskip
\section{Projection-based index reduction\label{sec:projection}}
The infinite horizon optimal control problem to find $u^*\in {\cal U}^+\equiv {\cal L}_2(0,\infty;\mathbb{R}^m)$ and the corresponding $y^*$ such that 
\begin{equation}
\label{eq:Opt0}
        \mathcal{E}(    {x}^0)\equiv\min_{\substack{{        u} \in {\cal U}^+ \\ {    {x}_1}(0) = {    {x}}^0 \\ {    {x}_1}(\infty) = {        0}}}\dfrac{1}{2}\int_0^{\infty}||    {y}(t)||^2+\dfrac{1}{\eta}||    {u}(t)||^2dt
        =\dfrac{1}{2}\int_0^{\infty}||    {y}^*(t)||^2+\dfrac{1}{\eta}||    {u}^*(t)||^2dt,
\end{equation}
subject to DAE constraints of the form (\ref{eq:NaSt_DAE}) 
\begin{equation*}
\begin{split}     
    \begin{pmatrix}
     {E}_{11} &     {0}\\
     {0} &     {0}
\end{pmatrix} 
\begin{pmatrix}
 \dot{    {x}}_1 \\
 \dot{    {x}}_2 
\end{pmatrix} &=
\begin{pmatrix}
     {A}_{11} &     {A}_{12}\\
     {A}_{12}^\top &     {0}
\end{pmatrix} 
\begin{pmatrix}
     {x}_1 \\
     {x}_2 
\end{pmatrix}
+\begin{pmatrix}
     {N} \\
        {0}   
\end{pmatrix}\begin{pmatrix}
     {x}_1 \otimes     {x}_1 
\end{pmatrix}+\begin{pmatrix}
     {B}_1 \\
     {B}_2 
\end{pmatrix}    {u}\\
   {y}&=    {C}_1    {x}_1,
\end{split}
\end{equation*}
{  where $x_1(0)={x}^0 \in \mathbb{R}^{n_1}$ is the initial condition of the above DAE}.  This is the analogue of the future energy function (\ref{eq:futureEnergy1}) extended to our DAE (\ref{eq:NaSt_DAE}).  To apply the results from the previous section, we need to reformulate (\ref{eq:NaSt_DAE}) as an s-free DAE.  Since system (\ref{eq:NaSt_DAE}) can arise from finite element discretizations of the Navier-Stokes equations, a discrete version of the Leray projector, e.g.~\cite{temam1977ns}, is natural.  Similar approaches were used in~\cite{HeinkandSorenProj,BennerIndex2,Miro} and here, we follow the approach in~\cite{HeinkandSorenProj} with the discrete projector
\begin{equation}
      {\Pi}= {I}_{n_1}- {A}_{12}\left( {A}_{12}^\top {E}_{11}^{-1} {A}_{12}\right)^{-1} {A}_{12}^\top {E}_{11}^{-1}.
\end{equation}
        % &\text{with}\\
%\end{align*}
% We now rewrite (\ref{eq:NaSt_DAE}) as: 
% \begin{align*}
%       {E}_{11}\dot{  {x}}_1&=  {A}_{11}  {x}_1+  {A}_{12}  {x_2}+  {N}\kronF{  {x}_1}{2}+  {B}_1  {u}\\
%       {A}_{12}^\top  {x}_1&+  {B}_2  {u}=0\\
%       {y}&=  {C_1}  {x}_1   
% \end{align*}
% and we split $x_1$ as:
We also make use of the following decomposition of this (oblique complementary) projector, e.g.~\cite{brust2020computationally}, as
\begin{displaymath}
    \Pi = \Theta_\ell \Theta_r^\top, \qquad \mbox{such that} \qquad \Theta_\ell^\top \Theta_r = I_{n_1-n_2}.
\end{displaymath}
We now write $x_1$ in its discrete Helmholtz decomposition of a discretely divergence-free ($x_f$) part and a discrete gradient term ($x_g$) that arises from the $B_2$ term in (\ref{eq:NaSt_DAE}),
\begin{equation}
   \label{eq:split}
        {x}_1=   x_f+   {x}_g ,
\end{equation}
where 
\begin{equation}
\label{eq:split22}
  {A}_{12}^\top   x_f=   {0}
  \qquad \mbox{and} \qquad {x}_g=-   {E}_{11}^{-1}   {A}_{12}\left(   {A}_{12}^\top   {E}_{11}^{-1}   {A}_{12}\right)^{-1}{B}_2   {u}.
\end{equation}
Using (\ref{eq:split}) and (\ref{eq:split22}), we can rewrite (\ref{eq:NaSt_DAE}) as
\begin{align}
\label{eq:DAE_2}
   {E}_{11}\dot{x}_f&=   {A}_{11}   x_f+   {A}_{12}   {x_2}+   {N}\kronF{   {x}_1}{2}+\left(   {B}_1-   {A}_{11}   {E}_{11}^{-1}   {A}_{12}\left(   {A}_{12}^\top   {E}_{11}^{-1}   {A}_{12}\right)^{-1}   {B}_2\right)   {u}\\
   &\qquad -{A}_{12}\left(   {A}_{12}^\top   {E}_{11}^{-1}   {A}_{12}\right)^{-1}   {B}_2\dot{   {u}}\nonumber \\
   0&={A}_{12}^\top   x_f\\
   \label{eq:output}
   {y}&=   {C_1}   x_f-   {C_1}   {E}_{11}^{-1}   {A}_{12}\left(   {A}_{12}^\top   {E}_{11}^{-1}   {A}_{12}\right)^{-1}   {B}_2   {u}.
\end{align}
The term $\dot{u}$ in (\ref{eq:DAE_2}) is missing when $B_2=0$, but in the $B_2\neq 0$ case, $\dot{u}$ can be eliminated by multiplying both sides by ${\Pi}$. For this case, the equation (\ref{eq:DAE_2}) becomes
\begin{align*}  
   {\Pi}   {E}_{11}\dot{x}_f&=   {\Pi}   {A}_{11}   x_f+   {\Pi}   {A}_{12}   {x_2}+   {\Pi}   {N}\kronF{   {x}_1}{2}+   {\Pi}\left(   {B}_1-   {A}_{11}   {E}_{11}^{-1}   {A}_{12}\left(   {A}_{12}^\top   {E}_{11}^{-1}   {A}_{12}\right)^{-1}   {B}_2\right)   {u}\\&\quad-   {\Pi}   {A}_{12}\left(   {A}_{12}^\top   {E}_{11}^{-1}   {A}_{12}\right)^{-1}   {B}_2\dot{   {u}}.
\end{align*}
Now, since $-   {\Pi}   {A}_{12}\left(   {A}_{12}^\top   {E}_{11}^{-1}   {A}_{12}\right)^{-1}   {B}_2=0$ and $   {\Pi}   {A}_{12}=0$, we have
\begin{align*}  
   {\Pi}   {E}_{11}\dot{x}_f&=   {\Pi}   {A}_{11}   x_f+   {\Pi}   {N}\kronF{   {x}_1}{2}+   {\Pi}   {B}   {u}
\end{align*}
where $   {B}=\left(   {B}_1-   {A}_{11}   {E}_{11}^{-1}   {A}_{12}\left(   {A}_{12}^\top   {E}_{11}^{-1}   {A}_{12}\right)^{-1}   {B}_2\right)$.\\
To express everything in terms of the $x_f$ variable, we expand $   {\kronF{   {x}_1}{2}}$ as 
\begin{align*}
\begin{split}
   \kronF{   {x}_1}{2}=&~\kronF{   x_f}{2}+\kronF{   {x}_g}{2}+   x_f\otimes   {x}_g+   {x}_g\otimes   x_f\\=&~\kronF{   x_f}{2}+\kronF{(   {E}_{11}^{-1}   {A}_{12}\left(   {A}_{12}^\top   {E}_{11}^{-1}   {A}_{12}\right)^{-1}   {B}_2)}{2}   \kronF{u}{2}-(   {E}_{11}^{-1}   {A}_{12}\left(   {A}_{12}^\top   {E}_{11}^{-1}   {A}_{12}\right)^{-1}   {B}_2\otimes   I_{n_1})  (u\otimes{x_f})\\
   &~-(   I_{n_1}\otimes   {E}_{11}^{-1}   {A}_{12}\left(   {A}_{12}^\top   {E}_{11}^{-1}   {A}_{12}\right)^{-1}   {B}_2)   (x_f\otimes{u}).  
\end{split}
\end{align*}
For notational simplicity, we define $    {G}=    {E}_{11}^{-1}    {A}_{12}\left(    {A}_{12}^\top    {E}_{11}^{-1}    {A}_{12}\right)^{-1}    {B}_2$ and use the fact that $    {\Pi}^\top    x_f=    x_f$ to arrive at
\begin{align}
\label{eq:DAE_3}
    {\Pi}    {E}_{11}    {\Pi}^\top\dot{    {x}}_f&=    {\Pi}    {A}_{11}    {\Pi}^\top    x_f+    {\Pi}    {N}\kronF{(    {\Pi}^\top    x_f)}{2}+    {\Pi N}\kronF{    {G}}{2}    \kronF{u}{2}-    {\Pi N}(    {Gu}\otimes     {\Pi}^\top    x_f)    \\&~ -    {\Pi N}(    {\Pi}^\top    x_f\otimes     {Gu})  +    {\Pi}    {B}    {u}. \nonumber
\end{align}
We define $x_d=    {\Theta}_l^\top    x_f$, multiply (\ref{eq:DAE_3}) by $ {\Theta}_r$, and use the fact that $    {\Theta}_r^\top    {\Pi}=    {\Theta}_r^\top$ to get 
\begin{align}
\label{eq:DAE_4}
    {\Theta}_r^\top    {E}_{11}    {\Theta}_r \dot{    {x}}_d=& ~   {\Theta}_r^\top    {A}_{11}    {\Theta}_r{    {x}}_d+    {\Theta}_r^\top    {N}\kronF{    {\Theta}_r}{2}\kronF{{    {x}}_d}{2}-    {    {\Theta}_r^\top N}
(    {G}\otimes     {\Theta}_r)(    {u}\otimes {    {x}}_d)    \\
&-    {    {\Theta}_r^\top N}(    {\Theta}_r\otimes     {G})({    {x}}_d\otimes     {u})  +    {    {\Theta}_r^\top N}\kronF{G}{2}    \kronF{u}{2}+    {\Theta}_r^\top    {B}    {u}. \nonumber
\end{align}
Since we can show that $\Theta_r^\top E_{11}\Theta_r$ is non-singular, equation (\ref{eq:DAE_4}) can be turned to an ODE in $\mathbb{R}^{n_1-n_2}$ and can also be written as
\begin{align}
\label{eq:DAE_5}
    {E}_d\dot{{    {x}}}_d&=    {A}_d{    {x}}_d+    {N}_d\kronF{{    {x}}_d}{2}+    H_1(x_d\otimes{u})+H_2(u\otimes{x_d})+    S_d    \kronF{u}{2}+B_du,
\end{align} 
where 
\begin{align*}
{E}_d&={\Theta}_r^\top  {E}_{11}{\Theta}_r & {A}_d&={\Theta}_r^\top{A}_{11}     {\Theta}_r\\
{N}_d&={\Theta}_r^\top{N}\kronF{{\Theta}_r}{2} &
S_d&=\Theta_r^\top N\kronF{G}{2} \\
H_1&=-\Theta^\top_rN(\Theta_r\otimes{G}) & H_2&=-\Theta^\top_rN(G\otimes{\Theta_r})\\
B_d&=\Theta_r^\top B.
\end{align*}
On the other hand, the algebraic part is given by 
\begin{equation} 
\label{eq:Dae_constraint1}   
x_2=    {F}_{a}(    x_d,    {u},\dot{u})
\end{equation}
and
\begin{equation}
\label{eq:Dae_constraint2}  
    {x}_g=-   {E}_{11}^{-1}   {A}_{12}\left(   {A}_{12}^\top   {E}_{11}^{-1}   {A}_{12}\right)^{-1}{B}_2   {u} = -{G} {u}, 
\end{equation}
where ${F}_{a}$ is a linear function in terms of $x_f$, $u$ and $\dot{u}$. Both (\ref{eq:Dae_constraint1}) and (\ref{eq:Dae_constraint2}) satisfy the assumptions presented in Section \ref{sec:II.D} but are irrelevant to the optimal control derivation. The output equation can be written in terms of $x_d$
\begin{equation*}
    {y}=   C_d{x}_d+   D_d{u}  
\end{equation*}
where $    {C}_d=    {C}_1    \Theta_r$ and $D_d=-{C_1}   {E}_{11}^{-1}   {A}_{12}\left(   {A}_{12}^\top   {E}_{11}^{-1}   {A}_{12}\right)^{-1}   {B}_2$. It can be easily checked that the obtained system is an s-free system that also satisfies Hypothesis 1 with $\mu=1$ and also satisfies the assumptions in Section \ref{sec:II.D}. We can now reformulate the DAE optimal control problem as an ODE optimal control problem with decoupled constraints given by (\ref{eq:Dae_constraint1}) and (\ref{eq:Dae_constraint2}). Since $x_d$ represents the dynamical states of the DAE system, we will only consider the state to output component of the output that is equivalent to (\ref{eq:NaSt_DAE}) with $y=C_1x_1-D_du$ instead of $y=C_1x_1$.  The s-free DAE optimal control problem is now defined as
\begin{equation}
\begin{split}
\label{eq:Opt1}
        \mathcal{E}(    x_d^0)\equiv\min_{\substack{u \in {\cal U}^+ \\ x_d(0) = x_d^0 \\ x_d(\infty) = 0}}\dfrac{1}{2}\int_0^{\infty}||  {y}(t)||^2+\dfrac{1}{\eta}||    {u}(t)||^2dt
        =\dfrac{1}{2}\int_0^{\infty}||    {y}^*(t)||^2+\dfrac{1}{\eta}||    {u}^*(t)||^2dt,
        \end{split}
\end{equation}
subject to the following s-free DAE constraints, 
\begin{equation}
\label{eq:opt}
\begin{split}
    {E}_d\dot{{    {x}}}_d&=    {A}_d{    {x}}_d+    {N}_d\kronF{{    {x}}_d}{2}+    H_1(x_d\otimes{u})+H_2(u\otimes{x_d})+    {S}_d    \kronF{u}{2}+B_d u\\
    x_2&=    {F}_{a}(    x_d,    {u},\dot{u})\\
    {x}_g&=-   {E}_{11}^{-1}   {A}_{12}\left(   {A}_{12}^\top   {E}_{11}^{-1}   {A}_{12}\right)^{-1}{B}_2   {u} \\
    {y}&=   C_d{x}_d,           
\end{split}
\end{equation} 
{ where $x_d^0=\Theta_l^\top x^0 \in \mathbb{R}^{n_1-n_2}$ is the initial condition of the dynamical part of the s-free DAE}. Since the value function depends only on $x_d$, the dynamical part of the state of the system, the derived optimal control problem in (\ref{eq:Opt1}) and (\ref{eq:opt}) turns out to be a future energy function problem as described in Section \ref{sec:II.B}. { We observe that the challenge in the resulting ODE is the presence of bilinear terms and a quadratic control. To tackle this issue, we use the polynomial-quadratic regulator (PQR) framework developed in \cite{borggaard2021PQR}. PQR provides $u^*_\degd(t)$, a degree $\degd$  polynomial approximation of $u^*(t)$, the solution of the optimal control described in (\ref{eq:Opt1}) and (\ref{eq:opt}).}

\section{Monolithic Sparse Formulation}
The issue with the strangeness-free formulation, presented in Section~\ref{sec:projection}, is that the sparsity of the matrices ($E_{11}$, $A_{11}$, $N$, and even $B_1$), if available, is destroyed when applying the projections.  Therefore, while our approach is very effective for smaller problems, it is impractical for large systems. To overcome this issue, we propose a formulation for larger problems that computes the energy function coefficients using the original (sparse) matrices. For simplicity of this discussion, we will consider (\ref{eq:NaSt_DAE}) with $B_2=0$. The projection procedure in Section~\ref{sec:projection} leads to the following system
\begin{equation}
\label{eq:DAE_5_B_2=0}
\begin{split}
    {E}_d\dot{{    {x}}}_d&=    {A}_d{    {x}}_d+    {N}_d\kronF{{    {x}}_d}{2}+    B_d    {u}  \\ 
    {y}&=   C_d{x}_d, 
\end{split}
\end{equation} 
where 
\begin{displaymath}
 {E}_d={\Theta}_r^\top  {E}_{11}{\Theta}_r, \quad\!
{A}_d=\Theta_r^\top{A}_{11}     {\Theta}_r, \quad\!
{N}_d={\Theta}_r^\top{N}\kronF{{\Theta}_r}{2}, \quad\!
B_d={\Theta}_r^\top{B}_1\quad\!  \mbox{and} \quad\!  C_d= C_1{\Theta}_r
\end{displaymath}
with $E_d$ being  invertible. Therefore, the coefficients of the approximations of the energy functions satisfy Theorem 1 {\rm (\cite[Th.~7]{kramer2023nonlinear1})} and Theorem 2 {\rm (\cite[Th.~6]{kramer2023nonlinear1})}. As an example, we take the future energy function 
\begin{align}\label{eq:energy_p+1}
\mathcal{E}_{\degd+1}(x_d)&=\dfrac{1}{2}\sum_{i=2}^{\degd+1} w_i^\top \kronF{x_d}{i},
\end{align}
and the associated optimal control 
\begin{align}\label{eq:control_p}
u_\degd^*(x_d)&=-\dfrac{\eta}{2}B_d^\top E_d^{-\top}\sum_{i=1}^{\degd} (i+1)W_{i+1}\kronF{x_d}{i},
\end{align}
where the coefficients $    {w}_k = {\tt vec}(    {W}_k) \in \mathbb{R}^{n^k}$ for $2\leq k\leq \degd+1$ satisfy the following.  For $k=2$, $W_2$ satisfies the Riccati equation
\begin{equation}\label{eq:pro_riccati}
 0=    {A}_d^\top E_d^{-\top}       {W}_2+    {W}_2E_d^{-1}    {A}_d+    {C}_d^\top        {C}_d-\eta     {W}_2    E_d^{-1}{B}_d    {B}_d^\top     E_d^{-\top}{W}_2,     
\end{equation}
and for $k>2$, the coefficients $w_k$ solve the linear systems 
\begin{equation}\label{eq:pro_LS}
\begin{split}
\mathcal{L}_k((   E_d^{-1} {A}_d -\eta     E_d^{-1} {B}_d    {B}_d^\top E_d^{-\top}    {W}_2 )^\top)    {w}_k=
\ -\mathcal{L}_{k-1}(    (E_d^{-1}{N}_d)^\top )    {w}_{k-1}\\
+\cfrac{\eta}{4}\sum_{\substack{i,j>2\\ i+j=k+2}} ij \, {\tt vec}(    {W}_i^\top       E_d^{-1} {B}_d    {B}_d^\top E_d^{-\top}        {W}_j),
\end{split}
\end{equation}
along with additional symmetrization of the result. As mentioned above, if we are dealing with large-scale problems, the resulting (\ref{eq:pro_riccati}) and (\ref{eq:pro_LS}) are dense equations and not computationally tractable. To resolve this issue, we propose an approach based on undoing the projection and performing all calculations using the original matrices to preserve the sparsity if initially available. 
Since $E_d$ is invertible, we define $\widetilde{W}_k$ for $2\leq k\leq \degd$ as follows
\begin{displaymath}
    {W}_k={E}_{d}^{\top}  \widetilde{W}_k\kronF{E_{d}}{k-1},
\end{displaymath} 
using the Kronecker-vec property, we also have
\begin{align*}
 {w}_k&=\kronF{({E}_{d}^\top)}{k}  {\tilde{w}}_k.
\end{align*}
Applying these identities to (\ref{eq:pro_riccati}) and (\ref{eq:pro_LS}) leads to the following
\begin{equation}\label{eq:pro_riccati_2}
 0=    {A}_d^\top        \widetilde{W}_2E_d+   E_d^{\top} \widetilde{W}_2    {A}_d+    {C}_d^\top        {C}_d-\eta     E_d^\top \widetilde{W}_2    {B}_d    {B}_d^\top  \widetilde{W}_2 E_d     
\end{equation}
and, using the extended definition of ${\cal L}_k^E(M)$ in (\ref{eq:calL}),
\begin{equation}\label{eq:pro_LS_2}
\begin{split}
\mathcal{L}^{{E}_{d}^\top}_k({A}_{d}^\top-\eta {E}_{d}^{\top} {\widetilde{W}}_2 {B}_d {B}_d^\top){ {\tilde{w}}}_k=-\mathcal{L}^{ {E}_{d}^\top}_{k-1}({ {N}}_{d}^\top) {\tilde{w}}_{k-1}\hspace{1in}\\ \hspace{1in}
+\cfrac{\eta}{4}\sum_{\substack{i,j>2\\ i+j=k+2}} ij \, {\tt vec}(    (\kronF{E_d}{i-1})^\top\widetilde{W}_i^\top        {B}_d    {B}_d^\top        \widetilde{W}_j\kronF{E_d}{j-1}).  
\end{split}
\end{equation}
We will now factor $\Theta_r$ from the left and right to recover the original matrices, resulting in the following 
\begin{equation}\label{eq:sparse_riccati}
 0=    \Theta_r^\top({A}_{11}^\top        \widehat{W}_2E_{11}+   E_{11}^{\top} \widehat{W}_2    {A}_{11}+    {C}_1^\top        {C}_1-\eta     E_{11}^\top \widehat{W}_2    {B}_1    {B}_1^\top  \widehat{W}_2 E_{11})\Theta_r     
\end{equation}
and
\begin{equation}\label{eq:sparse_LS}
\begin{split}
(\kronF{\Theta_r}{k})^\top\mathcal{L}^{ {E}_{11}^\top}_k({A}_{11}^\top-\eta  {E}_{11}^{\top} {\widehat{W}}_2 {B}_1 {B}_1^\top){ {\hat{w}}}_k=(\kronF{\Theta_r}{k})^\top\Big(-\mathcal{L}^{{E}_{11}^\top}_{k-1}({N}^\top) {\hat{w}}_{k-1}\hspace{.2in}\\
\hspace{.7in}+\cfrac{\eta}{4}\sum_{\substack{i,j>2\\ i+j=k+2}}ij \, {\tt vec}(    (\kronF{E_{11}}{i-1})^\top\widehat{W}_i^\top        {B}_1    {B}_1^\top        \widehat{W}_j\kronF{E_{11}}{j-1})\Big),  
\end{split}
\end{equation}
where $\widehat{W}_2=\Theta_r {\widetilde{W}}_2\Theta_r^\top$ and $\hat{w}_k=\kronF{\Theta_r}{k}\tilde{w}_k$ for $3\leq k \leq \degd$. The solutions of the above linear systems should also be symmetrized. The resulting Riccati equation in (\ref{eq:sparse_riccati}) is a projected Riccati equation and the authors in \cite{BennerIndex2} presented an inexact low-rank Newton-ADI method for solving large-scale algebraic Riccati equations associated with DAEs of differential index-2. We can also generalize their approach to the case where $B_2\neq 0$ in (\ref{eq:NaSt_DAE}). We note that the algorithm in \cite{BennerIndex2} only requires the original matrices $A_{11}$, $E_{11}$, $B_1$, $B_2$, and $C_1$, in other words, it is a sparsity-preserving algorithm. In \cite{kramer2023nonlinear1}, it was shown that the linear systems in (\ref{eq:pro_LS}) each have a unique solution. Therefore, the linear systems introduced in (\ref{eq:sparse_LS}), which are reformulations of the linear systems in (\ref{eq:pro_LS}), also have unique solutions. We now develop the main result that explains how to manage the projections in the {  k-way Lyapunov }equations in (\ref{eq:sparse_LS}). Before our main theorem, we will introduce the following notation
\begin{definition}\label{def1}
    Let $k$ be a positive integer, and $A$ and $B$ be matrices of dimensions $n\times n$ and $n\times r$, respectively, we define $\mathcal{M}_k^{B}(A)$ as
\begin{equation}\label{eq:lyap_block}
\mathcal{M}_k^{B}(A)=\begin{bmatrix}
    \underbrace{A\otimes{I}  \otimes{I}\otimes{\cdots}\otimes{I}}_{k \text{ terms}} ~ & \underbrace{B\otimes{A}\otimes{I}\otimes{\cdots} \otimes{I}}_{k \text{ terms}} & \cdots &\underbrace{ B\otimes{B}\otimes{\cdots}\otimes{B}\otimes{A}}_{k \text{ terms}}
    \end{bmatrix}.
\end{equation}
We drop the superscript if $B=I_n$, where $I_n$ is the $n\times n$ identity matrix.
\end{definition}
The theorem at the heart of our contribution to the monolithic sparse formulation is as follows
\begin{theorem}\label{sparse_LS_theorem}
If $\hat{w}_k \in \mathbb{R}^{n_1^k}$, with $k>2$, is the unique solution of the linear system in (\ref{eq:sparse_LS}) then it is also uniquely determined by the solution of the following augmented linear system 
\begin{align}\label{eq:sparse_LS_3}
\begin{split}
   \begin{pmatrix}
    \mathcal{L}_k^{ {E}_{11}^\top}(A_c^\top) & \mathcal{M}_k^{\tilde{I}}(A_{12})  \\
     \left(\mathcal{M}_k^{\tilde{I}}(A_{12})\right)^\top  & 0
\end{pmatrix}
\begin{pmatrix}
     {\hat{w}}_k\\
     {\Omega}\\
\end{pmatrix} =
\begin{pmatrix}
     {b}\\
     {0}\\
   \end{pmatrix},
\end{split}
\end{align}   
where
\begin{align*}
{b}&=-\mathcal{L}^{{E}_{11}^\top}_{k-1}({N}^\top) {\hat{w}}_{k-1}+\cfrac{\eta}{4}\sum_{\substack{i,j>2\\ i+j=k+2}} ij \, {\tt vec}(    (\kronF{E_{11}}{i-1})^\top\widehat{W}_i^\top        {B}_1    {B}_1^\top        \widehat{W}_j\kronF{E_{11}}{j-1})
\end{align*}   
and 
\begin{align*}
    A_c={A}_{11}-\eta B_1B_1^\top\widehat{W}_2{E}_{11}.
\end{align*}
Note that $\Omega$ is a dummy vector, $\mathcal{M}_k^{\tilde{I}}(A_{12})$ is defined in (\ref{eq:lyap_block}) and $\tilde{I} \in \mathbb{R}^{n_1\times (n_1-n_2)}$ is constructed from the columns of $I_{n_1}$ corresponding to the position of a set of non-basis columns of the column space of $A_{12}^\top$.
\end{theorem}
The proof of the theorem is presented in the appendix
\begin{remark}
    For the case where $B_2\neq 0$, following the development in (\ref{eq:split})--(\ref{eq:output}), the coefficients can be solved using the same linear system as in (\ref{eq:sparse_LS_3}) with a different right hand side $b$. 
\end{remark}
\begin{remark}
    For the case where $B_2=0$, we can express the energy function $\mathcal{E}_{\degd+1}(x_d)$ defined in (\ref{eq:energy_p+1}) in terms of the original $x_1$ variables. We start with the fact that for $2\leq i\leq \degd+1$, $w_i$ is given by
    \begin{align*}
w_i^\top=\hat{w}_i^\top\kronF{(E_{11}\Theta_r)}{i}.
    \end{align*}
    Then
    \begin{align*}
        \mathcal{E}_{\degd+1}(x_d)&=\dfrac{1}{2}\sum_{i=2}^{\degd+1} w_i^\top \kronF{x_d}{i}=\dfrac{1}{2}\sum_{i=2}^{\degd+1} \hat{w}_i^\top \kronF{(E_{11}\Theta_rx_d)}{i}.
    \end{align*}
    We also have that $\Theta_r x_d=\Pi^\top x_f=x_f=x_1$, which leads to
    \begin{align*}
        \mathcal{E}_{\degd+1}(x_1)=\dfrac{1}{2}\sum_{i=2}^{\degd+1} \hat{w}_i^\top \kronF{(E_{11}x_1)}{i}.
    \end{align*}
     Similarly, to express the optimal degree $\degd$ control $u^*_\degd(x_d)$, defined in (\ref{eq:control_p}), in terms of the original $x_1$ and $B_1$,
     we follow the same derivation to get
     \begin{align*}
         u_\degd^*(x_d)&=-\dfrac{\eta}{2}B_d^\top E_d^{-\top}\sum_{i=1}^{\degd}(i+1)W_{i+1}\kronF{x_d}{i}
         =-\dfrac{\eta}{2}B_d^\top E_d^{-\top}\sum_{i=1}^{\degd}(i+1)E_d^\top \Theta_l^\top \widehat{W}_{i+1}\kronF{\Theta_l}{i}\kronF{E_d}{i}\kronF{x_d}{i}\\
         &=-\dfrac{\eta}{2}B_d^\top \sum_{i=1}^{\degd}(i+1)\Theta_l^\top \widehat{W}_{i+1}\kronF{\Theta_l}{i}\kronF{E_d}{i}\kronF{x_d}{i}\\
         &=-\dfrac{\eta}{2}B_1^\top \sum_{i=1}^{\degd}(i+1)\Theta_r\Theta_l^\top \widehat{W}_{i+1}\kronF{\Theta_l}{i}(\kronF{\Theta_r}{i})^\top\kronF{E_{11}}{i}\kronF{\Theta_r}{i}\kronF{x_d}{i}.
         \end{align*}
        Hence
    \begin{align*}
            u_\degd^*(x_1)=-\dfrac{\eta}{2}B_1^\top \sum_{i=1}^{\degd}(i+1)\Pi^\top \widehat{W}_{i+1}\kronF{\Pi}{i}\kronF{E_{11}}{i}\kronF{x_1}{i}. 
    \end{align*}
     Finally, using the fact that $\Pi^\top \widehat{W}_{i+1}\kronF{\Pi}{i}=\widehat{W}_{i+1}$, we get
     \begin{equation}
          u_\degd^*(x_1) =-\dfrac{\eta}{2}B_1^\top \sum_{i=1}^{\degd}(i+1) \widehat{W}_{i+1}\kronF{E_{11}}{i}\kronF{x_1}{i}.
     \end{equation}
     So, we can express both $u_\degd^*$ and $\mathcal{E}_{\degd+1}$ in terms of $x_1$ and the original matrices. For the case when $B_2\neq 0$, $u^*_\degd$ and $\mathcal{E}_{\degd+1}$ can be expressed in terms of the original matrices and $x_f$.
\end{remark}
\begin{remark}
  Without any algebraic constraints, systems of the form $E\dot{x} = Ax + N\kronF{x}{2} + Bu$ with invertible $E$ would satisfy the equation
  \begin{equation}\label{eq:PQR}
     \mathcal{L}_k^{E^\top}( A^\top-\eta {E}^\top {W}_2 BB^\top) \hat{w}_k = b
  \end{equation}
  with ${w}_k=\kronF{(E^\top)}{k} \hat{w}_k$, which can be seen by applying results from \cite{kramer2023nonlinear1} to the system $\dot{x} = E^{-1} Ax + E^{-1}N\kronF{x}{2} + E^{-1}Bu$, then multiplying the equation for the coefficients $\hat{w}_k$ by $\kronF{(E^\top)}{k}$.  Thus, equation (\ref{eq:sparse_LS_3}) can be seen as a natural generalization of (\ref{eq:PQR}) to the space of tensors where each mode-$i$ { matricization} of the tensor $W_k$ is in the kernel of $A_{12}^\top$.  This fact could also lead to the development of new algorithms.  
  \end{remark}
  
  \begin{remark}{ For the case where $E$ is singular, the system matrices $A$, $N$ and $B$ will dictate the s-index and there is no guarantees that an explicit decoupling of the DAE exists.  In principle, if there is a transformation to an s-free system with different structure, the development presented here could be applicable.
  }\end{remark}
\section{Numerical Results}
\subsection{Scalar example}
For the first example, we will approximate the optimal control $u^*(t)$, with a degree $\degd$ polynomial for a system of the form (\ref{eq:NaSt_DAE}) with $n_1=2$,  $n_2=1$.
\begin{align*}
   \dot{z}_1 +\dot{z}_2 &= z_1 + z_3 + \frac{1}{2}z_1^2 - 2z_1z_2 + \frac{1}{2}z_2^2 + u \\
   \dot{z}_2 &= -2z_2 + z_3 \\
   0         &= z_1 + z_2 \\
   y         &= z_2,
\end{align*}
which is in the form (\ref{eq:NaSt_DAE}) with
\begin{align*}
\begin{split}
    x_1&=\begin{pmatrix}z_1\\z_2\end{pmatrix}, \quad
    x_2=\begin{pmatrix}z_3\end{pmatrix},
    \quad
    E_{11}=\begin{pmatrix}
        1 &\phantom{-}1\\
        0 &\phantom{-}1
    \end{pmatrix},
    \quad 
    A_{11}=\begin{pmatrix}
        1 & \phantom{-}0\\
        0 &-2
    \end{pmatrix},
    \quad
    A_{12}=\begin{pmatrix}
        1\\
        1
    \end{pmatrix}, 
    \\
    N&=\begin{pmatrix}
         0.5 & -1 & -1 &  \phantom{-}0.5\\
         0    &     \phantom{-}0    &     \phantom{-}0      &   \phantom{-}0
    \end{pmatrix},\quad
    B_1=\begin{pmatrix}
        1\\
        0
    \end{pmatrix}, 
    \quad
    B_2=(0), \quad \text{ and } \quad
    C_1=\begin{pmatrix}
        0&1
    \end{pmatrix}.  
\end{split}    
\end{align*}
Note the symmetric form of the quadratic coefficient term $N$.
We compute $u^*_\degd$, polynomial feedback laws  of degrees $\degd=1,2,3,4,$ and $ 5$ with $\eta=10$. {  The quality of the approximation is determined by simulating the closed loop s-free DAE system with two different starting values $x_d^0=1$ and $x_d^0=-1$  and comparing the integral form of the value function (\ref{eq:Opt1}) approximated with an upper limit of $T=100$ and $u^*_\degd$ instead of $\infty$ and $u^*$, with its degree $\degd+1$ polynomial approximation (\ref{eq:energy_p+1}) evaluated at $x_d^0$}. The results are presented in Tables \ref{tab:case11} and \ref{tab:case12}. As expected, increasing the degree of approximation tends to give better results, minimizing both control cost and its associated energy/value function approximation error.

\begin{table}[ht]
\centering
\caption{\label{tab:case11} Relative and absolute errors versus the degree of approximation for the scalar example with $x_d^0=-1$. }

\begin{tabular}{ |c|c|c|c|c| } 
\hline
 $\degd$ & Value function & Integral form & Absolute error & {Relative error}    \\ 
 \hline\hline
 $1$& $0.057916$ & $ 0.108050$ & $5.01 \times 10^{-2}$&  $46.400\%$ \\  
 \hline
 $2$& $0.082611$ & $0.092197$ & $9.59 \times 10^{-3}$& $10.397\%$ \\
 \hline
 $3$ & $0.090320$ & $0.091300$& $9.80 \times 10^{-4}$& $\phantom{0}1.073\%$ \\
 \hline
  $4$ & $0.091509$ & $0.091260$& $2.50 \times 10^{-4}$& $\phantom{0}0.274\%$ \\
 \hline
  $5$ & $0.091223$ & $0.091275 $& $5.21 \times 10^{-5}$& $\phantom{0}0.057\%$ \\
 \hline
\end{tabular}
\end{table}

\begin{table}[ht]
\centering
\caption{\label{tab:case12} Relative and absolute errors versus the degree of approximation for the scalar example with $x_d^0=1$. }
\begin{tabular}{ |c|c|c|c|c| } 
    \hline
 $\degd$ & Value function & Integral form & Absolute error & Relative error    \\ 
 \hline\hline
  $1$& $0.057916$ &$0.041139$& $1.67\times 10^{-2}$&   $40.779\%$ \\  
 \hline
 $2$& $0.033220$ & $0.039961$& $6.74 \times 10^{-3}$& $16.868\%$ \\
 \hline
 $3$& $0.040929$ & $0.039676 $& $1.25 \times 10^{-3}$& $\phantom{0}3.1602\%$\\
 \hline
  $4$ & $0.039740$ & $0.039666$& $7.73 \times 10^{-5}$& $\phantom{0}0.186\%$ \\
 \hline
  $5$ & $0.039453$ & $0.039668$& $2.14\times 10^{-4}$& $\phantom{0}0.541\%$ \\
 \hline
\end{tabular}
\end{table}

\subsection{Discretized boundary control of the Fisher equation}
As a second example, we study Dirichlet boundary control of the 1D Fisher equation on the unit interval
\begin{align}
\label{eq:fis_bvp}
\begin{split}
 w_t(t,\xi)=\alpha w_{\xi\xi}(t,\xi)+\beta w(t,\xi)(1-w(t,\xi))\hspace{.7in} \\
 \mbox{with} \qquad w(t,0)=u(t) \qquad \mbox{and} \qquad
 w(t,1)=0, \qquad t>0,
\end{split}  
\end{align}
 that minimizes the absolute value of the solution average 
\begin{align}
\label{eq:fis_obj}
        J(u) = \dfrac{1}{2}\int_0^\infty \left(\int_0^1 w(t,\xi) d\xi\right)^2 + \dfrac{1}{\eta} |u(t)|^2 dt,
\end{align}
(this objective function does not guarantee $w$ is small, but we observed zero limiting solutions in the cases where the energy function was finite).
We use a Galerkin finite element approximation with piecewise-linear basis functions and a uniform mesh with $N_e$ elements to perform the spatial discretization of the BVP (\ref{eq:fis_bvp}) and the objective function (\ref{eq:fis_obj}) and seek to control the discretized system. The resulting optimal control problem is of the form (\ref{eq:NaSt_DAE}) with $n_1=N_e$, $n_2=1$ and
\begin{displaymath}
E_{11}=M, \text{ and } A_{11}=-\alpha K+\beta M,
\end{displaymath}
where $M$ and $K$ are the mass and stiffness matrices resulting from the finite element approximation, $A_{12} = e_1$ is the first unit vector in $\mathbb{R}^{n_1}$,
\begin{displaymath}
B_1=0 \in \mathbb{R}^{n_1}, \quad
B_2=-1 \in \mathbb{R}^{n_2}, 
\end{displaymath}
and $C_1$ is the row vector where each component is $\dfrac{1}{N_e}$.
The matrix $N$ results from the quadratic term $-\beta w^2$ in (\ref{eq:fis_bvp}). The discretization of the objective function in (\ref{eq:fis_obj}) gives the following. 
\begin{equation}
\label{eq:Ex2_0bj}
  J(u) \approx J_h(u) = \dfrac{1}{2}\int_0^\infty  |C_1x_1(t)|^2 + \dfrac{1}{\eta}|u(t)|^2 dt.  
\end{equation}
Therefore, the value function is also the future energy function, and the discretized problem is an optimal control problem of the form described by (\ref{eq:NaSt_DAE}) and (\ref{eq:Opt0}).

\subsubsection{Case 1}
For the first case study, we will use the resulting system described above with $n_1=16$, $\alpha=0.1$, $\beta=3$, $\eta=0.1$.  For these parameter values, $A_{11}$ has two unstable eigenvalues.  This system is converted to an s-free DAE using the procedure described above. We compute $u_\degd^*$ for $\degd=1,2$ and $3$ and use a random initial condition ${x^0_d}$ in $[-0.3,0.3]^{15}$ given by
\begin{eqnarray*}
x_d^0&\approx& [
  \begin{matrix} 
   -0.038& 
   -0.284&
    0.030& 
   -0.039& 
   -0.048& 
   -0.101& 
   -0.177&
    0.072&  \cdots \end{matrix}\\
  && \hspace{1in} 
   \begin{matrix} \cdots & 
   -0.120& 
   -0.140& 
    0.073& 
    0.018& 
   -0.219& 
    0.008& 
   -0.189 \end{matrix}
 ]^\top  
\end{eqnarray*} 
(using the script {\tt rng(2); 0.3*2*(rand(n1-n2,1)-1/2*ones(n1-n2,1);})).
Note that this initial condition is not in the domain of attraction of the controlled system with linear state feedback.  However, both quadratic and cubic feedback laws stabilize the system from this initial condition $x^0_d$.  {  We also compare
the integral form of the value function (\ref{eq:Opt1}) approximated with an upper limit of $T=100$ and $u^*_\degd$ instead of $\infty$ and $u^*$, with its degree $\degd+1$ polynomial approximation (\ref{eq:energy_p+1}) evaluated at $x_d^0$}. Table \ref{tab:case3} shows a significant improvement in both control cost (integral form) as well as agreement between actual and predicted values (integral form and value function at $x_d^0$) when using cubic feedback over quadratic feedback for this example. 
\begin{table}[ht]
\centering
\caption{\label{tab:case3}Absolute and relative errors versus approximation degree for Dirichlet control of Fisher equation.}
\begin{tabular}{ |c|c|c|c|c| } 
 \hline
 $\degd$ & Value function & Integral form & Absolute error & Relative error \\ 
 \hline\hline
 $1$& $0.0689 $ & divergence & divergence&    divergence \\    
 \hline
 $2$& $0.1219$ & $0.1589$& $ 3.69 \times 10^{-2} $& $23.26 \%$ \\
 \hline
 $3$& $0.1461$ & $0.1562$& $ 1.00 \times 10^{-2}$& $\phantom{0}6.43 \%$ \\
 \hline
\end{tabular}
\end{table}

\subsubsection{Case 2}
In this second case, we sweep over a set of random initial conditions to get a general sense of how much larger the stability region grows with higher degree feedback laws.   For this reason, we consider the same system as in case 1 with $n_1=16, \alpha=0.1, \beta=1, \eta=1$ and we perform a study with 1000 different random initial conditions in the range $[-1,1]^{15}$. The number of times the closed-loop system is
unstable (or if was always stable) and the average relative error for each degree of approximation are reported in Table \ref{tab:case4}. Note that we did not include the unstable results
from the polynomial approximations in the reported averages
(since those simulations would not finish).
\begin{table}[ht] 
\centering
\caption{\label{tab:case4}Stability and average relative error of closed-loop systems from sampled initial guesses.}
\begin{tabular}{ |c|c|c|c| } 
 \hline
 Degree & 1&2&3 \\ 
 \hline\hline
 Average relative error & $50.46 \%$ &$31.09\%$ &$5.61 \%$  \\ 
 \hline
  Stability & Unstable 24/1000 & Stable & Stable  \\    
 \hline
\end{tabular}
\end{table}
In both of these test cases, we show that increasing the degree of the polynomial approximation to the optimal control gives better results in both predicting the optimal control cost (energy function) and the ability to stabilize equilibrium points. 
\subsubsection{Case 3}
In this last case, we will show the effectiveness of the method on a larger scale problem. We study the same problem with a different objective function and a fixed initial condition.  The Dirichlet boundary control of the 1D Fisher equation
on the unit interval
\begin{align}
\begin{split}
 w_t(t,\xi)=\dfrac{1}{2}w_{\xi\xi}(t,\xi)+5 w(t,\xi)(1-w(t,\xi))\hspace{.7in} \\
 \mbox{with} \qquad w(t,0)=u(t) \qquad \mbox{and} \qquad
 w(t,1)=0, \qquad t>0,\\
 w(0,x)=0.05\cos(2\pi x) \hspace{.7in} x\in [0,1]
\end{split}  
\end{align}
 that minimizes the value of the solution average 
\begin{align}
\label{eq:fis_obj_1}
        J(u) = \dfrac{1}{2}\int_0^\infty \left(\int_0^1 w^2(t,\xi) d\xi\right) + \dfrac{1}{\eta} u^\top(t)u(t)\, dt,
\end{align}
with $\eta=10^{-3}$. We pick $n_1=1024$ and consider both the state to output and input to output components. When $k=3$, the resulting linear system described in equation (\ref{eq:sparse_LS_3}) that is part of Theorem \ref{sparse_LS_theorem}, has size $2n_1^3-(n_1-n_2)^3=1,076,884,481$. The system was solved using GMRES(10) and a preconditioner. In this case, we compute $u^*_\degd$, the polynomial feedback laws with $\degd=1,2$.  For performance, we also compare
the integral form of the value function (\ref{eq:Opt1}) approximated with an upper limit of $T=100$ and $u^*_\degd$ instead of $\infty$ and $u^*$, with its degree $\degd+1$ polynomial approximation (\ref{eq:energy_p+1}) evaluated at $x_d^0$. The results are reported in Table \ref{tab:case5}, and for this case as well, we clearly see that the quadratic controller performs better than the linear controller in both the control cost and the agreement between actual and predicted values.

\begin{table}[ht]
\caption{\label{tab:case5}Absolute and relative errors versus approximation degree for Dirichlet control of Fisher equation.}
\centering
\begin{tabular}{|c|c|c|c|c|} 
 \hline
 $\degd$ & Value function & Integral form & Absolute error & Relative error \\ 
 \hline\hline
 $1$& $0.0036601 $ &  $0.016362$  & $0.012702$&    $77.631 \%$  \\    
 \hline
 $2$& $0.0059891$ & $0.006221$& $ 0.000232 $& $\phantom{0}3.726 \%$ \\
 \hline
\end{tabular}
\end{table}

\section{Conclusions and Future Work}
We have now established two strategies for computing polynomial coefficients for future and past energy function approximations when the input-output system is described by a DAE with a Stokes-type structure.  The first approach uses the strangeness-framework for DAEs to convert the system to an s-free form from which the algorithms used in \cite{kramer2023nonlinear1} can be applied (with some modifications to include feedthrough or quadratic inputs in the most general cases).  The effectiveness of this approach was demonstrated through two numerical studies.  The first example was of the case $B_2=0$ and showed benefits of adding higher degree feedback terms, but also showed that even quadratic feedback laws can expand the stability radius about an equilibrium point.  The second example was the discretization of a Dirichlet boundary control problem where the Dirichlet boundary condition is imposed as a constraint and leads to a Stokes-type DAE with $B_2\neq 0$.  The reduction to an s-free system and subsequent calculation of polynomial feedback laws up to degree 3 echoed the improvements seen in the first example.

{  We note that dense matrices were effectively needed to achieve this index reduction, which limits the applicability of this method to relatively small problem sizes, depending on the approximation degree desired.  To alleviate this, we developed the second approach, a monolithic sparse formulation, that while larger, works with the original matrices. If these are sparse, then matrix-free solvers are attractive and should open up the applicability to much larger systems. The effectiveness of the sparse approach was demonstrated in the last numerical experiment, a Dirichlet boundary control problem using the semi-discretized Fisher equation with $1024$ degrees of freedom.}

\section{Appendix: proof of Theorem \ref{sparse_LS_theorem}}
Before starting the proof we will introduce and prove the following lemmas
\begin{lemma}\label{lemma_comb}
    For any positive integers $k$, $r_2$, and $n_1$ with $n_1>r_2$, we have the following identity
\begin{equation}\label{eq:identity1}
    r_2\sum_{i=1}^{k} n_1^{k-i}(n_1-r_2)^{i-1}=n_1^k-(n_1-r_2)^k  . 
    \end{equation}
\end{lemma}
\begin{proof}
    We will proceed by  induction on $k$. For $k=1$, we have
    \begin{align*}
        r_2\sum_{i=1}^{1} n_1^{1-i}(n_1-r_2)^{i-1}=r_2= n_1-(n_1-r_2).
    \end{align*}
    We now assume that the expression is true for positive integer $k$, thus
    \begin{align*}
    r_2\sum_{i=1}^{k} n_1^{k-i}(n_1-r_2)^{i-1}=n_1^k-(n_1-r_2)^k.   
    \end{align*}
    and show that it holds for $k+1$. Consider the sum below where 
    we isolate a term and re-index
    \begin{align*}
    \sum_{i=1}^{k+1} n_1^{k+1-i}(n_1-r_2)^{i-1}=n_1^{k}+\sum_{i=2}^{k+1} n_1^{k+1-i}(n_1-r_2)^{i-1}=n_1^{k}+\sum_{i=1}^{k} n_1^{k-i}(n_1-r_2)^i.
    \end{align*}
    Multiplying both sides by $r_2$
    \begin{align*}
    r_2\sum_{i=1}^{k+1} n_1^{k+1-i}(n_1-r_2)^{i-1}=r_2n_1^{k}+r_2(n_1-r_2)\sum_{i=1}^{k} n_1^{k-i}(n_1-r_2)^{i-1}
    \end{align*}
    and using the inductive hypothesis leads to
    \begin{align*}
    r_2\sum_{i=1}^{k+1} n_1^{k+1-i}(n_1-r_2)^{i-1}=r_2n_1^{k}+(n_1-r_2)(n_1^k-(n_1-r_2)^k)=n_1^{k+1}-(n_1-r_2)^{k+1}.
    \end{align*}
This proves the identity for $k+1$ and completes the proof.
\end{proof}
\begin{lemma}\label{lemma_R}
    Let $A_{12} \in \mathbb{R}^{n_1\times n_2}$ be a matrix with rank $r_2$ and $n_1>n_2$ and $\tilde{I} \in \mathbb{R}^{n_1\times (n_1-r_2)}$ is constructed from the columns of $I_{n_1}$ corresponding to a non-basis set of columns for the column space of $A_{12}^\top$. Then, we have that
    \begin{align*}
        \operatorname{rank}(\widetilde{R})=n_1-r_2,
    \end{align*}
    where $\widetilde{R}=R\tilde{I}$ with $R=I_{n_1}-A_{12}A^\dagger_{12}$.
\end{lemma}
\begin{proof}
    Let $x \in \mathbb{R}^{n_1-r_2}$ such that $\widetilde{R}x=0$. Hence,
    \begin{align*}
        \widetilde{R}x=0 \implies R\tilde{x}=0,
    \end{align*} 
    with $\tilde{x}=\tilde{I}x$. Therefore,
    \begin{align*}
     \tilde{x} \in \operatorname{col}(A_{12}) \implies \tilde{x}=A_{12}y_1
    \end{align*}
for some $y_1 \in \mathbb{R}^{n_2}$. From the definition of $\tilde{x}$, we also have that
\begin{align*}
    \tilde{x}\in \operatorname{col}(\tilde{I}). 
\end{align*} 
Now let $\Pi_1$ be a permutation matrix such that $A^\top_{12}\Pi_1=\begin{pmatrix}
    N_1&M_1
\end{pmatrix}$, where $N_1\in \mathbb{R}^{n_2\times r_2}$ and $M_1\in \mathbb{R}^{n_2\times n_1-r_2}$ with $\operatorname{rank}(N_1)=r_2$. This permutation ensures that the first $r_2$ columns of $A^\top_{12}\Pi_1$ are a basis for the column space of $A^\top_{12}$. Using the permutation defined above, we get 
\begin{align*}
    \Pi_1^\top\tilde{x}=\Pi_1^\top A_{12}y_1=\begin{pmatrix}
        N_1^\top y_1\\
        M_1^\top y_1
    \end{pmatrix}=\Pi_1^\top\tilde{I}x=\begin{pmatrix}
        0\\
        I_{n_1-r_2}
    \end{pmatrix}x=\begin{pmatrix}
        0\\
        x
    \end{pmatrix},
\end{align*}
which implies that $N_1^\top y_1=0$ and $M_1^\top y_1=x$. Now since the columns of $M_1$ are linear combinations of the columns of $N_1$, we have that $M_1=N_1H$, for some matrix $H$. This leads to $H^\top N_1^\top y_1=0=M_1^\top y_1=x$. Hence $x=0$, which implies that $\widetilde{R}$ is full rank with $\operatorname{rank}(\widetilde{R})=n_1-r_2$ and completes the proof.
\end{proof}

\begin{lemma}\label{lemma_rank}
    For any integer $k$ and  $A_{12}\in \mathbb{R}^{n_1\times n_2}$, both matrices  $\mathcal{M}_k(A_{12})\in \mathbb{R}^{n_1^k\times kn_2n_1^{k-1}}$ and $\mathcal{M}_k^{\tilde{I}}(A_{12}) \in \mathbb{R}^{n_1^k\times d_{2}}$, given in Definition~\ref{def1},  have the same rank
    \begin{align*}
        {\rm rank}(\mathcal{M}_k^{\tilde{I}}(A_{12}))={\rm rank}(\mathcal{M}_k(A_{12}))=n_1^k-(n_1-r_2)^k  
    \end{align*}
where $r_2={\rm rank}(A_{12})$, $d_2
=n_2\sum_{i=1}^{k}n_1^{k-1}(n_1-r_2)^{i-1}$ and $\tilde{I} \in \mathbb{R}^{n_1\times (n_1-r_2)}$ is constructed from the columns of $I_{n_1}$ corresponding to a non-basis set of columns for the column  space of $A_{12}^\top$. When $A_{12}$ is full column rank, then $(\mathcal{M}_k^{\tilde{I}}(A_{12}))$ is also full column rank with
\begin{align*}
        \text{rank}(\mathcal{M}_k^{\tilde{I}}(A_{12}))=\text{rank}(\mathcal{M}_k(A_{12}))=d_2=n_1^k-(n_1-n_2)^k  
    \end{align*}
    Remark: even when $A_{12}$ is full column rank, $\mathcal{M}_k(A_{12})$ is not.
\end{lemma}
\begin{proof}
For both rank calculations we will use the identity in {\rm (\cite[Th.~5]{Matsaglia01011974})} 
\begin{equation}\label{eq:rank_id}
\text{rank}([A~~B])=\text{rank}(A)+\text{rank}(B-AA^\dagger B)= \text{rank}(B)+\text{rank}(A-BB^\dagger A)
  \end{equation}
where $A^\dagger$ and $B^\dagger$ are pseudo-inverses of $A$ and $B$, respectively. Using (\ref{eq:rank_id}) on  $\mathcal{M}_k(A_{12})$ results in 
\begin{align*}
\begin{split}
    \text{rank}(\mathcal{M}_k(A_{12}))&=\text{rank}\left(\begin{bmatrix}\underbrace{A_{12}\otimes{I_{n_1}}\otimes \cdots \otimes{I_{n_1}}}_{k\text{ times}}\end{bmatrix}\right)\\&+\text{rank}\left(\begin{bmatrix}R\otimes{\underbrace{A_{12}\otimes{I_{n_1}}\otimes{\cdots}\otimes{I_{n_1}}}_{k-1 \text{ times}}}|\cdots| R\otimes{\underbrace{I_{n_1}\otimes{I_{n_1}}\otimes{\cdots}\otimes A_{12}}_{k-1\text{ times}}}
    \end{bmatrix}\right)\\
    &=r_2n_1^{k-1}\\&+\text{rank}\left(\begin{bmatrix}R\otimes{\underbrace{A_{12}\otimes{I_{n_1}}\otimes{\cdots}\otimes{I_{n_1}}}_{k-1 \text{ times}}}|\cdots| R\otimes{\underbrace{I_{n_1}\otimes{I_{n_1}}\otimes{\cdots}\otimes A_{12}}_{k-1\text{ times}}}
    \end{bmatrix}\right),
\end{split}
\end{align*}
where \begin{equation}\label{eq:R}
    R=I_{n_1}-A_{12}A_{12}^\dagger.
\end{equation}
with $\text{rank}(R)=r_{R}=n_1-r_2$. Using the same identity over the $k-1$ remaining blocks  and Lemma~\ref{lemma_comb}, we get
\begin{align}
    \text{rank}(\mathcal{M}_k(A_{12}))=r_2\sum_{i=1}^kn_1^{k-i}r_R^{i-1} =r_2\sum_{i=1}^kn_1^{k-i}(n_1-r_2)^{i-1} =n_1^k-(n_1-r_2)^k.
\end{align}
  Applying the same approach on $\mathcal{M}_k^{\tilde{I}}(A_{12})$ results in a similar expression 
\begin{align*}
    \text{rank}(\mathcal{M}^{\tilde{I}}_k(A_{12}))=r_2\sum_{i=1}^kn_1^{k-i}r_{\widetilde{R}}^{i-1}
\end{align*}  
where $\widetilde{R}=R\tilde{I}$  with ${\rm rank}(\widetilde{R})=n_1-r_2$, shown in Lemma~\ref{lemma_R}. Therefore
\begin{align}
    \text{rank}(\mathcal{M}_k^{\tilde{I}}(A_{12}))=r_2\sum_{i=1}^kn_1^{k-i}(n_1-r_2)^{i-1}=n_1^k-(n_1-r_2)^k,
\end{align}
where we have again resorted to Lemma~\ref{lemma_comb}.  Thus the ranks of both matrices are the same and equal to $n_1^k-(n_1-r_2)^k$.
If $A_{12}$ has full column rank, we can replace $r_2$ with $n_2$ to complete the proof.
\end{proof}
\begin{remark}\label{remark_rank}
    An important observation to make is because of the structure of the two matrices, the columns of $\mathcal{M}^{\tilde{I}}_k(A_{12})$ are a sub-collection of the columns of $\mathcal{M}_k(A_{12})$ and when $A_{12}$ is full rank, they constitute a basis for the column space of $\mathcal{M}^{\tilde{I}}_k(A_{12})$. 
\end{remark}
\begin{lemma}\label{lemma_per}
    Let $k,n_1$ and $n_2$ be a fixed positive integer and $\mathcal{M}_k^{R}(A_{12})$ given in Definition~\ref{def1} where $A_{12}\in \mathbb{R}^{n_1\times n_2}$ with $n_1>n_2$. If we define $\hat{Z}^k_i$ to be the $i^\text{th}$ column block of $\mathcal{M}_k^{R}(A_{12})$ and $P_{\hat{Z}^k_i}=I_{n_1^k}-\hat{Z}^k_i(\hat{Z}_i^k)^\dagger$, then the following relationship holds 
\begin{displaymath} P_{\hat{Z}^k_1}P_{\hat{Z}^k_2}\cdots P_{\hat{Z}^k_k}=\kronF{R}{k}
\end{displaymath}
with $R=I_{n_1}-A_{12}A_{12}^\dagger$.
\end{lemma}
\begin{proof}
We will again rely on induction.  For $k=1$, 
\begin{displaymath}
    P_{\hat{Z}_1^1}=I_{n_1}-A_{12}A_{12}^\dagger \qquad \mbox{and} \qquad \kronF{R}{1}=I_{n_1}-A_{12}A_{12}^\dagger.
\end{displaymath}
Now we assume that the statement
\begin{align*}
     P_{\hat{Z}^k_1}P_{\hat{Z}^k_2}\cdots P_{\hat{Z}^k_k}=\kronF{R}{k}
\end{align*}
 is true for $k$ and will show that it is also true for $k+1$. Note that
 \begin{align*}
   A_{12}A_{12}^\dagger R=A_{12}A_{12}^\dagger(I_{n_1}-A_{12}A_{12}^\dagger)=A_{12}A_{12}^\dagger-A_{12}A_{12}^\dagger=0\\ 
   RA_{12}A_{12}^\dagger =(I_{n_1}-A_{12}A_{12}^\dagger)A_{12}A_{12}^\dagger=A_{12}A_{12}^\dagger-A_{12}A_{12}^\dagger=0,
 \end{align*} 
 which implies that
 \begin{align*} 
 \hat{Z}_i^{k+1}(\hat{Z}_i^{k+1})^\dagger\hat{Z}_j^{k+1}(\hat{Z}^{k+1}_j)^\dagger=0\text{  when  } i\neq j
 \end{align*}
 and all the cross terms in the product $P_{\hat{Z}_1^{k+1}}P_{\hat{Z}_2^{k+1}}\cdots P_{\hat{Z}_{k+1}^{k+1}}$ vanish. Therefore, we have 
 \begin{align*}
P_{\hat{Z}_1^{k+1}}P_{\hat{Z}_2^{k+1}}\cdots P_{\hat{Z}_{k+1}^{k+1}}&=I_{n_1^{k+1}}-\hat{Z}_1^{k+1}(\hat{Z}_1^{k+1})^\dagger-\hat{Z}_2^{k+1}(\hat{Z}_2^{k+1})^\dagger-\cdots-\hat{Z}_{k+1}^{k+1}(\hat{Z}_{k+1}^{k+1})^\dagger   
 \end{align*} 
 and the same relationship holds for the induction hypothesis
\begin{align*}
 P_{\hat{Z}_1^{k}}P_{\hat{Z}_2^{k}}\cdots P_{\hat{Z}_{k}^{k}}&=I_{n_1^k}-\hat{Z}_1^{k}(\hat{Z}_1^{k})^\dagger-\hat{Z}_2^{k}(\hat{Z}_2^{k})^\dagger-\cdots-\hat{Z}_{k}^{k}(\hat{Z}_{k}^{k})^\dagger.   
 \end{align*}
 From the above equations and the definitions of $\hat{Z}_i^k$ and $\hat{Z}_i^{k+1}$, we observe that 
 \begin{align*}
     &I_{n_1^{k+1}}-\hat{Z}_1^{k+1}(\hat{Z}_1^{k+1})^\dagger-\hat{Z}_2^{k+1}(\hat{Z}_2^{k+1})^\dagger-\cdots-\hat{Z}_{k+1}^{k+1}(\hat{Z}_{k+1}^{k+1})^\dagger\\
     =&(I_{n_1^k}-\hat{Z}_1^{k}(\hat{Z}_1^{k})^\dagger-\hat{Z}_2^{k}(\hat{Z}_2^{k})^\dagger-\cdots-\hat{Z}_{k}^{k}(\hat{Z}_{k}^{k})^\dagger)\otimes{I_{n_1}})-(\kronF{R}{k}(\kronF{R}{k})^\dagger\otimes{A_{12}A_{12}^\dagger}).
 \end{align*}
Hence
\begin{align*}
P_{\hat{Z}_1^{k+1}}P_{\hat{Z}_2^{k+1}}\cdots P_{\hat{Z}_{k+1}^{k+1}}&=( P_{\hat{Z}_1^{k}}P_{\hat{Z}_2^{k}}\cdots P_{\hat{Z}_{k}^{k}}\otimes{I_{n_1}})-((\kronF{R}{k}(\kronF{R}{k})^\dagger\otimes{A_{12}A_{12}^\dagger}).
\end{align*}
Now, the induction hypothesis and the fact that $RR^\dagger=RR=R$ implies that 
\begin{align*}
P_{\hat{Z}_1^{k+1}}P_{\hat{Z}_2^{k+1}}\cdots P_{\hat{Z}_{k+1}^{k+1}}&=(\kronF{R}{k}\otimes{I_{n_1}})-(\kronF{R}{k}\otimes{A_{12}A_{12}^\dagger})\\
&=\kronF{R}{k}(I_{n_1}-{A_{12}A_{12}^\dagger})=\kronF{R}{k}R=\kronF{R}{k+1}.
\end{align*}
This proves the $k+1$ case and completes the proof.
\end{proof}
We are now ready to prove Theorem \ref{sparse_LS_theorem} 
\begin{proof}
    Let $k$ be a fixed integer and $\hat{w}_k$ be the unique solution of the linear system (\ref{eq:sparse_LS}). We define
% \begin{align*}
%     b&= -\mathcal{L}^{{E}_{11}^\top}_{k-1}({N}^\top) {\hat{w}}_{k-1}
% +\cfrac{\eta}{4}\sum_{\substack{i,j>2\\ i+j=k+2}} ij \, {\tt vec}(    (\kronF{E_{11}}{i-1})^\top\widehat{W}_i^\top        {B}_1    {B}_1^\top        \widehat{W}_j\kronF{E_{11}}{j-1})\\
\begin{align*}
    Z&=\mathcal{L}^{ {E}_{11}^\top}_k({A}_{11}^\top-\eta  {E}_{11}^{\top} {\widehat{W}}_2 {B}_1{ {B}_1}^\top){ {\hat{w}}}_k-b.
\end{align*}
Then equation (\ref{eq:sparse_LS}) becomes $(\kronF{\Theta_r}{k})^\top Z=0$. Multiplying both sides by $\kronF{\Theta}{k}_\ell$, combining the products of Kronecker products, and using the factorization of $\Pi$ results in
\begin{align*}
    \kronF{\Pi}{k} Z=0 \implies \kronF{(I-\Pi_1)}{k}Z=0,
\end{align*}
where $\Pi_1=A_{12}(A_{12}^\top E_{11}A_{12})^{-1}A_{12}^\top E_{11}^{-1}$ from the definition of $\Pi$. Hence, there exist vector coefficients $\gamma_i$ such that
\begin{align*}
    Z=-\sum_{i=1}^{k}Z^k_i\gamma_i
\end{align*}
where $Z^k_i$ is the $i^\text{th}$ column block of $\mathcal{M}_k(A_{12})$. Therefore
\begin{equation}\label{eq:block1}
    \mathcal{L}^{ {E}_{11}^\top}_k({A}_{11}^\top-\eta  {E}_{11}^{\top} {\widehat{W}}_2 {B}_1{ {B}_1}^\top){ {\hat{w}}}_k+\sum_{i=1}^kZ^k_i\gamma_i=b.
\end{equation}
Now since $\Pi^\top\Theta_r=\Theta_r$ and $\hat{w}_k=\kronF{\Theta_r}{k}\tilde{w}_k$, it implies that $(\kronF{\Pi}{k})^\top\hat{w}_k=\hat{w}_k$ and along with $A_{12}^\top\Pi^\top=0$, we can deduce that 
\begin{equation}\label{eq:block2}
    (Z_i^k)^\top\hat{w}_k=0
\end{equation} 
for all $1\leq i \leq k$.  Combining both (\ref{eq:block1}) and (\ref{eq:block2}) leads to the
following augmented linear system 
\begin{equation}\label{eq:sparse_LS_2}
\begin{split}
   \mathcal{G}_k\begin{pmatrix}
     {\hat{w}}_k\\
     {\Gamma}\\
\end{pmatrix}=\begin{pmatrix}
    \mathcal{L}_k^{ {E}_{11}^\top}( {A}_c^\top) & \mathcal{M}_k(A_{12})  \\
     \left(\mathcal{M}_k(A_{12})\right)^\top & 0
\end{pmatrix}
\begin{pmatrix}
     {\hat{w}}_k\\
     {\Gamma}\\
\end{pmatrix} =
\begin{pmatrix}
     {b}\\
     {0}\\
   \end{pmatrix}
\end{split}
\end{equation}   
where $A_c={A}_{11}-\eta B_1B_1^\top  {\widehat{W}}_2{E}_{11}$ and $\Gamma$, the dummy variable, is a vector that combines the $\gamma_i$ vectors. Now since $\mathcal{M}_k(A_{12})$ is not full rank, then the matrix above (\ref{eq:sparse_LS_2}) is a singular matrix. We now want to compute the rank of this latter by using the following identity \cite{Matsaglia01011974}
\begin{align}\label{eq:id_rank}
    \text{rank}\Big(\begin{bmatrix}
        A & B\\
        C & 0
    \end{bmatrix}\Big)=\text{rank}(B)+\text{rank}(C)+\text{rank}[(I-BB^\dagger)A(I-C^\dagger C)]
\end{align}
where $B^\dagger$ and $C^\dagger$ are pseudo-inverses of $B$ and $C$, respectively. We have that
\begin{align*}
    &\text{rank}(\mathcal{G}_k)=\text{rank}\left(\begin{pmatrix}
    \mathcal{L}_k^{ {E}_{11}^\top}( {A}_c^\top) & \mathcal{\widehat{M}}_k(A_{12})  \\
     \left(\mathcal{\widehat{M}}_k(A_{12})\right)^\top & 0
\end{pmatrix}\right)
\end{align*}
where $\mathcal{\widehat{M}}_k(A_{12})=\begin{pmatrix}
   Z_k^k ~Z_{k-1}^k~\cdots~Z_1^k 
\end{pmatrix}$ comes from reversing the order of the column blocks in $\mathcal{{M}}_k(A_{12})$. Since this is a reordering operation than it is rank preserving. Now, using (\ref{eq:id_rank}) on the reordered  matrix we get
\begin{align*}
&\text{rank}\left(\mathcal{G}_k\right)=2\times\text{rank}(Z_1^k)+\text{rank}\left(\begin{pmatrix}
    P_{Z_1^k} & 0\\
    0 & I
\end{pmatrix}\begin{pmatrix}
    \mathcal{L}_k^{ {E}_{11}^\top}( {A}_c^\top) & Z^k_k &\cdots &Z^k_{2}  \\
     (Z_k^k)^\top & 0& \cdots &0\\
     \vdots & \vdots & \vdots &\vdots\\
     (Z_{2}^k)^\top & 0& \cdots &0\\
\end{pmatrix}\begin{pmatrix}
    P_{Z_1^k} & 0\\
    0 & I
\end{pmatrix}\right)\\
&=2n_2(n_1-n_2)^{k-1}+\text{rank}\left(\begin{pmatrix}
    P_{Z^k_1}\mathcal{L}_k^{ {E}_{11}^\top}( {A}_c^\top)P_{Z^k_1} & P_{Z^k_1}Z^k_k &\cdots &P_{Z^k_1}Z^k_2  \\
     (Z_k^k)^\top P_{Z^k_1} & 0& \cdots &0\\
     \vdots & \vdots & \vdots &\vdots\\
     (Z_2^k)^\top P_{Z^k_1} & 0& \cdots &0\\
\end{pmatrix}\right)\\
&=2n_2(n_1-n_2)^{k-1}+\text{rank}\left(\begin{pmatrix}
    P_{Z_1^k}\mathcal{L}_k^{ {E}_{11}^\top}( {A}_c^\top)P_{Z_1^k}& \hat{Z}^k_k &\cdots &\hat{Z}^k_2 \\  (\hat{Z}^k_k)^\top & 0& \cdots &0\\
     \vdots & \vdots & \vdots &\vdots\\
    (\hat{Z}^k_2)^\top & 0& \cdots &0\\
\end{pmatrix}\right)
\end{align*}
where $P_{Z^k_1}=(I-Z^k_1 (Z^k_1)^\dagger)$ and $\hat{Z}^k_i$ is defined in Lemma~\ref{lemma_per}. Using the same identity in (\ref{eq:id_rank}) for the remaining $k-1$ blocks implies that
\begin{align*}
  \text{rank}(\mathcal{G}_k)
=2n_2\sum_{i=1}^kn_1^{k-i}(n_1-n_2)^{i-
1}
+\text{rank}\Big(P_{\hat{Z}_k^k}\cdots P_{\hat{Z}_2^k}P_{\hat{Z}_1^k}\mathcal{L}_k^{ {E}_{11}^\top}({A}_c^\top)P_{\hat{Z}_1^k}P_{\hat{Z}_2^k}\cdots P_{\hat{Z}_k^k}\Big),
\end{align*}
where we also used the fact that $\hat{Z}^k_1=Z^k_1$ and $P_{\hat{Z}_i^k}\hat{Z}_j^k=\hat{Z}_j^k$ for $2<i<j$. Now, using Lemma~\ref{lemma_per} and the symmetric property of the orthogonal projections defined above, we get
\begin{align*}
  \text{rank}(\mathcal{G}_k)
=2n_2\sum_{i=1}^kn_1^{k-i}(n_1-n_2)^{i-
1}
+\text{rank}\Big(\kronF{R}{k}\mathcal{L}_k^{ {E}_{11}^\top}( A_c^\top)\kronF{R}{k}\Big).
\end{align*}
We also show that  
\begin{align*}
    \Pi^\top R&= ({I}_{n_1}-E_{11}^{-\top} A_{12}(A_{12}^\top E_{11}^{-\top}A_{12})^{-1}A_{12}^\top)(I_{n_1}-A_{12}A_{12}^\dagger)=(I_{n_1}-A_{12}A_{12}^\dagger)=R
\end{align*}
and
\begin{align*}
    R~\Pi^\top &= (I_{n_1}-A_{12}A_{12}^\dagger)({I}_{n_1}-E_{11}^{-\top} A_{12}(A_{12}^\top E_{11}^{-\top}A_{12})^{-1}A_{12}^\top)
    \\&=({I}_{n_1}-E_{11}^{-\top} A_{12}(A_{12}^\top E_{11}^{-\top}A_{12})^{-1}A_{12}^\top)=\Pi^\top,
\end{align*}
which leads to 
\begin{align*}
    &\text{ rank}\Big(\kronF{R}{k}\mathcal{L}_k^{ {E}_{11}^\top}( A_c^\top)\kronF{R}{k}\Big)
    =\text{ rank}\Big(\kronF{R}{k}\kronF{\Pi}{k}\mathcal{L}_k^{ {E}_{11}^\top}( A_c^\top)(\kronF{\Pi}{k})^\top\kronF{R}{k}\Big)\\
    =&\text{ rank}\Big(\kronF{R}{k}\kronF{\Theta_l}{k}\mathcal{L}^{E_d^\top}_k( {A_d^c}^\top)(\kronF{\Theta_l}{k})^\top\kronF{R}{k}\Big)
    =\text{ rank}\Big(\kronF{Q_1}{k}(\kronF{Q_1}{k})^\top\kronF{\Theta_l}{k}\mathcal{L}^{ E_d^\top}_k({A_d^c}^\top)(\kronF{\Theta_l}{k})^\top\kronF{Q_1}{k}(\kronF{Q_1}{k})^\top\Big),
\end{align*}
where $A_d^c=A_d-\eta B_dB_d^\top\widetilde{W}_2E_d$ and  $R=Q_1Q_1^\top$ is a low rank factorization of the orthogonal projector onto the null space of $A_{12}^\top$, that can be achieved using a QR-factorization. Now, since $Q_1\in \mathbb{R}^{n_1\times n_1-n_2}$ has full column rank, then
\begin{align*}
    \text{ rank}\Big(\kronF{R}{k}\mathcal{L}_k^{ {E}_{11}^\top}( A_c^\top)\kronF{R}{k}\Big)
    =\text{ rank}\Big((\kronF{Q_1}{k})^\top\kronF{\Theta_l}{k}\mathcal{L}^{ {E}_{d}^\top}_k({A_d^c}^\top)(\kronF{\Theta_l}{k})^\top\kronF{Q_1}{k}\Big).
\end{align*}
We also have that
\begin{align*}
    \Theta_r(\Theta_l)^\top Q_1Q_1^{\top}=\Pi^\top R=R,
\end{align*}
which implies 
\begin{align*}
\text{rank }(\Theta_r(\Theta_l)^\top Q_1Q_1^{\top})=\text{rank }(R) = n_1-n_2.
\end{align*}
Using the fact that both $\Theta_l$ and $Q_1$ have full column rank, we get
\begin{align*}
    \text{rank }((\Theta_l)^\top Q_1)= n_1-n_2.
\end{align*}
Therefore, the square matrix $(\Theta_l)^\top Q_1$ is invertible and 
\begin{align*}
    \text{ rank}\Big(\kronF{R}{k}\mathcal{L}_k^{ {E}_{11}^\top}( A_c^\top)\kronF{R}{k}\Big)=\text{rank}\Big(\mathcal{L}^{E_d^\top}_k({A_d^c}^\top)\Big).
\end{align*}
Now, as shown in {\rm \cite[Th.~6]{kramer2023nonlinear1}}, since $E_d^{-1}A_d^c$ is stable then $\mathcal{L}_k({A_d^c}^\top E_d^{-\top})$ is also stable. Therefore $\mathcal{L}^{E_d^\top}_k({A_d^c}^\top)$ is full rank.
\begin{align*}
\text{rank}\Big(\mathcal{L}^{E_d^\top}_k({A_d^c}^\top)\Big)=(n_1-n_2)^k
\end{align*}
So, using Lemma~\ref{lemma_comb} leads to
\begin{align*}
\text{rank}(\mathcal{G}_k)
=2(n_1^k-(n_1-n_2)^k)+(n_1-n_2)^k
=2n_1^k-(n_1-n_2)^k.
\end{align*}
This result confirms that $\mathcal{G}_k\in \mathbb{R}^{n_1^k+kn_2n_1^{k-1}\times n_1^k+kn_2n_1^{k-1}}$ is not invertible. This is where Lemma~\ref{lemma_rank} and remark~\ref{remark_rank} plays an important role, since replacing $\mathcal{M}_k(A_{12})$ with $\mathcal{M}_k^{\tilde{I}}(A_{12})$ in $\mathcal{G}_k$ will not change its rank,
but removes the redundant rows and columns which makes
\begin{displaymath}
\begin{pmatrix}
    \mathcal{L}_k^{ {E}_{11}^\top}(A_c^\top) & \mathcal{M}_k^{\tilde{I}}(A_{12})  \\
     \left(\mathcal{M}^{\tilde{I}}_k(A_{12})\right)^\top & 0
\end{pmatrix}\in \mathbb{R}^{2n_1^k-(n_1-n_2)^k\times 2n_1^k-(n_1-n_2)^k}    
\end{displaymath} 
a full rank invertible matrix. Thus, the following linear system 
\begin{align*}
    \begin{pmatrix}
    \mathcal{L}_k^{ {E}_{11}^\top}( A_c^\top) & \mathcal{M}_k^{\tilde{I}}(A_{12})  \\
     \left(\mathcal{M}_k^{\tilde{I}}(A_{12})\right)^\top  & 0
\end{pmatrix}
\begin{pmatrix}
     {\hat{w}}_k\\
     {\Omega}\\
\end{pmatrix} =
\begin{pmatrix}
     {b}\\
     {0}\\
   \end{pmatrix}
\end{align*}
where $\Omega$ is a dummy vector constructed from some entries of $\Gamma$. Since the matrix is invertible, this ensures the uniqueness of $\hat{w}_k$ and completes the proof.  
\end{proof}

\section{Declarations}
\subsection[]{Conflict of interest}
The authors declare no competing interests.
\subsection[]{Funding}
This work was partially supported by the NSF under Grant CMMI-2130727. 
\subsection{Author's contribution}
Both authors, Hamza Adjerid and Jeff Borggaard, worked on developing the method, writing the codes, running simulations, writing and proofreading the manuscript. 
\bibliography{sn-bibliography}

\end{document}